\documentclass[]{interact}

\usepackage{epstopdf}

\usepackage[numbers,sort&compress]{natbib}
\bibpunct[, ]{[}{]}{,}{n}{,}{,}
\renewcommand\bibfont{\fontsize{10}{12}\selectfont}

\theoremstyle{plain}
\newtheorem{theorem}{Theorem}[section]
\newtheorem{lemma}[theorem]{Lemma}

\newtheorem{proposition}[theorem]{Proposition}

\theoremstyle{definition}
\newtheorem{definition}[theorem]{Definition}
\newtheorem{example}[theorem]{Example}

\theoremstyle{remark}
\newtheorem{remark}{Remark}

\usepackage{amsmath, amssymb, graphicx, mathrsfs}
\usepackage{cases}
\usepackage{color}
\usepackage{amsthm}
\usepackage{amsfonts}
\usepackage{url}
\usepackage{sidecap}

\usepackage{color}
\usepackage{hyperref}
\definecolor{darkblue}{rgb}{0,0,0.8}
\hypersetup{colorlinks,breaklinks,linkcolor=darkblue,urlcolor=darkblue,anchorcolor=darkblue,citecolor=darkblue}

\def\bc{\mathbf{c}}

\def\be{\mathbf{e}}

\def\bw{\mathbf{w}}
\def\bg{\mathbf{g}}
\def\bp{\mathbf{p}}

\def\re{\mathbb{R}}

\def\bbG{\mathbb{G}}
\def\mbi{\mathbb{I}}

\def\bit{\begin{itemize}}
\def\eit{\end{itemize}}

\def\Niin{\mathcal{N}_i^{\mathrm{in}}}
\def\Niout{\mathcal{N}_i^{\mathrm{out}}}

\def\Nikin{\mathcal{N}_{ik}^{\mathrm{in}}}
\def\Nikout{\mathcal{N}_{ik}^{\mathrm{out}}}

\def\1{\mathbf{1}}
\def\0{\mathbf{0}}
\def\la{\langle}
\def\ra{\rangle}

\def\a{{\alpha}}
\def\b{\beta}
\def\g{{\gamma}}

\def\T{{\intercal}}

\def\diag{{\rm diag}}

\def\argmin{\mathop {\rm argmin}}

\newtheorem{assumption}{Assumption}

\usepackage{amsmath,amsthm,mathtools, amssymb, bbm, xspace}
\usepackage{booktabs,ragged2e}
\usepackage{multirow}
\usepackage{subcaption}
\usepackage{graphicx}

\usepackage{cases}

\def\det{\mathop{\hbox{\rm det}}}
\def\diag{\mathop{\hbox{\rm diag}}}

\def\spose#1{\hbox to 0pt{#1\hss}}

\def\text #1{\hbox{\quad#1\quad}}

\def\T{^T\!}

\def\bbG{\mathbb{G}}

\def\nthinsp{\mskip -2   mu}

\def\superstar{^{\raise 0.5pt\hbox{$\nthinsp *$}}}
\def\SUPERSTAR{^{\raise 0.5pt\hbox{$*$}}}

\def\lamstarT {\lambda^{\raise 0.5pt\hbox{$\nthinsp *$}T}}

\def\hbar{\skew{4.2}\bar h}

		\def\bkE{{\rm I\kern-.17em E}}
		\def\bk1{{\rm 1\kern-.17em l}}
		\def\bkD{{\rm I\kern-.17em D}}
		\def\bkR{{\rm I\kern-.17em R}}
		\def\bkP{{\rm I\kern-.17em P}}
		\def\bkY{{\bf \kern-.17em Y}}
		\def\bkZ{{\bf \kern-.17em Z}}

		\def\bc{\begin{center}}
		\def\be{\begin{enumerate}}
		\def\bi{\begin{itemize}}
		\def\ec{\end{center}}
		\def\ee{\end{enumerate}}
		\def\ei{\end{itemize}}
		\def\es{\end{small}}
		\def\eS{\end{slide}}
		\def\qed{\quad \vrule height7.5pt width4.17pt depth0pt}

	\def\cp2problem#1#2#3#4{\fbox
		 {\begin{tabular*}{0.9\textwidth}
			{@{}l@{\extracolsep{\fill}}l@{\extracolsep{6pt}}l@{\extracolsep{\fill}}c@{}}
				#1 & & $#4 $ 
			\end{tabular*}}}

		\def\bkE{{\rm I\kern-.17em E}}
		\def\bk1{{\rm 1\kern-.17em l}}
		\def\bkD{{\rm I\kern-.17em D}}
		\def\bkR{{\rm I\kern-.17em R}}
		\def\bkP{{\rm I\kern-.17em P}}
		
		\def\bkZ{{\bf{Z}}}

\newcommand {\beeq}[1]{\begin{equation}\label{#1}}
\newcommand {\eeeq}{\end{equation}}
\newcommand {\bea}{\begin{eqnarray}}
\newcommand {\eea}{\end{eqnarray}}

\def\texitem#1{\par\smallskip\noindent\hangindent 25pt
               \hbox to 25pt {\hss #1 ~}\ignorespaces}

\newcommand{\beq}{\begin{equation}}
\newcommand{\eeq}{\end{equation}}
\newcommand{\beqn}{\begin{eqnarray}}
\newcommand{\eeqn}{\end{eqnarray}}
\newcommand{\beqno}{\begin{eqnarray*}}
\newcommand{\eeqno}{\end{eqnarray*}}
\newcommand{\bma}{\begin{displaymath}}
\newcommand{\ema}{\end{displaymath}}
\newcommand{\bnu}{\begin{enumerate}}
\newcommand{\enu}{\end{enumerate}}
\newcommand{\bce}{\begin{center}}
\newcommand{\ece}{\end{center}}
\newcommand{\btb}{\begin{tabular}}
\newcommand{\etb}{\end{tabular}}

\def\G{{\mathbb{G}}}
\def\D{{\mathsf{D}}}

\def\K{{\mathsf{K}}}

\def\ba{{\mathbf{a}}}

\def\bp{{\mathbf{p}}}
\def\bx{{\mathbf{x}}}
\def\by{{\mathbf{y}}}
\def\bz{{\mathbf{z}}}

\def\b1{{\mathbf{1}}}

\def\diag{{\rm diag}}

\usepackage[normalem]{ulem} 

\def \qed {\hfill \vrule height6pt width 6pt depth 0pt}

\newcommand{\bs}{{\mathbf{s}}}

\begin{document}


\title{$AB$/Push-Pull Method for Distributed Optimization in Time-Varying Directed Networks}

\author{
\name{Angelia Nedi{\'c}\textsuperscript{a}\thanks{CONTACT Angelia Nedi{\'c}. Email: Angelia.Nedich@asu.edu. This work has been partially supported by the Office of Naval Research award N00014-21-1-2242} and Duong Thuy Anh Nguyen\textsuperscript{a} and Duong Tung Nguyen\textsuperscript{a}}
\affil{\textsuperscript{a}School of Electrical, Computer and Energy Engineering, Arizona State University, Tempe, AZ, United States}
}

\maketitle

\begin{abstract}
In this paper, we study the distributed optimization problem for a system of agents embedded in time-varying directed communication networks. Each agent has its own cost function and agents cooperate to determine the global decision that minimizes the summation of all individual cost functions. We consider the so-called push-pull gradient-based algorithm (termed as $AB$/Push-Pull) which employs both row- and column-stochastic weights simultaneously to track the optimal decision and the gradient of the global cost while ensuring consensus and optimality. We show that the algorithm converges linearly to the optimal solution over a time-varying directed network for a constant stepsize when the agent's cost function is smooth and strongly convex. The linear convergence of the method has been shown in Saadatniaki et al.\ (2020), where the multi-step consensus contraction parameters for row- and column-stochastic mixing matrices are not directly related to the underlying graph structure, and the explicit range for the stepsize value is not provided. With respect to Saadatniaki et al.\ (2020), the novelty of this work is twofold: (1) we establish the one-step consensus contraction for both row- and column-stochastic mixing matrices with the contraction parameters given explicitly in terms of the graph diameter and other graph properties; and (2) we provide explicit upper bounds for the stepsize value in terms of the properties of the cost functions, the mixing matrices, and the graph connectivity structure.

\end{abstract}



    \section{Introduction}
    We consider a system of $n$ agents embedded in a communication network with the goal to collaboratively solve the following minimization problem:
	\begin{equation} \label{eq-problem}
	\min_{x\in \re^p}~ f(x)=\frac{1}{n}\sum\limits_{i=1}^n f_i(x),
	\end{equation}
    where each function $f_i: \re^p\rightarrow \re$ represents the cost function of agent $i$, is strongly convex and known by agent $i$ only. 
    The strong convexity condition implies that problem~\eqref{eq-problem} has a unique optimal solution.
    The agents want to determine the optimal solution by performing local computations and limited information exchange with their local neighbors in the communication network. 
    Decentralized and collaborative approach is particularly appealing in large-scale, multi-agent systems with privacy concerns and limited computation, communication, or storage capabilities. In these scenarios, the data is collected and/or stored in a distributed manner, thus, computing tasks are distributed over all the agents and information exchange occurs only between the agents with direct communication links. Such problems appear in many engineering and scientific applications for example in wireless sensor networks \cite{Rabbat2004}, distributed sensing \cite{Bazerque2010}, trajectory optimization for formation control of vehicles~\cite{Stipanovic2002}, large-scale machine learning \cite{Tsianos2012ML}, and cooperative multi-agent systems \cite{Raffard2004}. 
    
    Distributed optimization of the sum of convex functions has been of considerable interest and many algorithms have been proposed including gradient-based methods \cite{Nedic2009,Nedic2011,Ram2009, Srivastava2011,xin2018linear}, dual averaging methods \cite{Duchi2012},
    ADMM \cite{Wei2014}, and 
    Newton methods \cite{mokhtari2017network,varagnolo2016newton}. Early works have often assumed that the underlying network is undirected (see literature review in~\cite{Huan2021}) and most commonly require doubly stochastic or weight-balanced \cite{Gharesifard2012} mixing matrices. Reference~\cite{Shi2015} uses a gradient difference structure in the algorithm to provide the first-order method that achieves a geometric convergence with the requirement of symmetric weights. Based on the ADMM approach, the work in~\cite{Wei2014} demonstrates a linear convergence while the Nesterov's acceleration method in reference~\cite{olshevsky2017linear} obtains convergence times that scale linearly in the number of agents. Reference~\cite{Shi2015_2} investigates decentralized algorithms that take advantage of proximal operations for the non-smooth terms. In \cite{pu2018distributed,pu2017flocking}, stochastic variants of distributed methods have been considered for asynchronous computations. 

    In many scenarios, agents communications are directed, for example, due to broadcasting at different power levels, thus resulting in communications that correspond to directed graphs. To cope with directed graphs, reference~\cite{Tsianos2012} introduces a subgradient-push algorithm to achieve consensus among the agents on an optimal point. The work in~\cite{nedic2015distributed} further studies the push-sum technique for time-varying directed graphs with a convergence rate of $O(\ln{t}/\sqrt{t})$ for diminishing stepsizes. Aiming to improve the convergence rate, algorithms ADD-OPT~\cite{xi2018add} and Push-DIGing~\cite{nedic2017achieving} incorporates the push-sum protocol with  gradient estimation approach, and show geometric convergence for a sufficiently small step-size. The implementation of these methods require the knowledge of agents' out-degree in order to construct a column-stochastic weight matrix, which is later removed in~\cite{xi2018linear} and in FROST method~\cite{Xin2019FROSTFastRO}. 
    
    The aforementioned push-sum based works use an independent algorithm to asymptotically compute the right or left eigenvectors of the weight matrix, corresponding to the eigenvalue of~$1$. Thus, the resulting algorithms are nonlinear and involve additional computation among agents. Unlike the push-sum protocol, the alternate AB/Push-Pull methods introduced in~\cite{pshi21, xin2018linear} use a row-stochastic matrix and a column-stochastic matrix simultaneously to achieve a linear convergence. Recent work in~\cite{Pu2020} further addresses the challenge of noisy information exchange and shows linear convergence (in expectation) to a neighborhood of the optimum exponentially fast, under a constant stepsize. The analysis of $AB$/Push-Pull (with stochastic gradients) was shown in~\cite{XinSahuKhanKar2019}. A variant of the method, where the stepsize $\a$ is agent dependent, has been analyzed in~\cite{pshi21} for the case of a static graph. All the aforementioned work on the AB/Push-Pull methods is for a static directed graph.
    The AB/Push-Pull method for time-varying directed graphs has been studied in~\cite{Saadatniaki2020}, where a linear convergence is shown for the case when the global objective function is smooth and strongly-convex, and the underlying time-varying graphs have bounded connectivity. In order to facilitate privacy design, the recent work in~\cite{Wang2022} proposes to tailor gradient methods for differentially-private distributed optimization. The work in~\cite{Huan2022} provides a general gradient-tracking based privacy-preserving algorithm with added randomness in optimization parameters and shows that the added randomness has no influence on the accuracy of optimization.

    In this paper, 
    we consider $AB$/Push-Pull algorithm where the agent communications are given by a sequence of time-varying directed graphs. At every time $k$, the agents' updates are described by two non-negative matrices that are compliant with the connectivity structure of the graph: a row-stochastic matrix for the mixing of the decision variables {\it (pull-step)} and a column-stochastic matrix for tracking the average gradients {\it (push-step)}. 
    We prove that the method converges to the optimal solution geometrically fast, provided that the stepsize is small enough and the agents' objective functions are smooth enough. Moreover, we provide an explicit upper bound for the stepsize range and characterize the convergence rate in terms of the problem parameters, algorithms' parameters (weight matrices), and the underlying graphs' connectivity structures. 
    
    A key difficulty in the analysis is imposed by the time-varying nature of the mixing matrices. 
    Our analysis makes use of time-varying weighted averages and time-varying weighted norms, where the weights are defined in terms of stochastic vector sequences associated with the mixing matrix sequences. This allows us to establish consensus contractions per each update step for both row- and column-stochastic mixing matrices.
    This is unlike the work in~\cite{Saadatniaki2020} that considers the $AB$/Push-Pull method over time-varying graphs, where the analysis makes use of the Euclidean norms -- at the expense of relying on a multi-step consensus contraction, even when every underlying graph is strongly connected. 
    Moreover, through the use of time-varying weighted norms and the relations of the weight matrices with the underlying graphs, we provide explicit upper bounds for the stepsize range in terms of properties of the mixing matrices and the graphs' connectivity structure. This is in sharp contrast with~\cite{Saadatniaki2020} where no explicit range is provided. Also, our analysis in this paper is in contrast with~\cite{XinSahuKhanKar2019,pshi21} where the stepsize range is given in terms of the singular values of the weight matrices, which are neither explicitly capturing the structure of the matrices nor the underlying graph connectivity structure. 
		
    The structure of this paper is as follows. We first provide notation, introduce our algorithm and state basic assumptions in Section~\ref{Sec:note-algo-asum}. We present some basic results in Section~\ref{sec-basic}. We establish the convergence properties of the algorithm in Section~\ref{sec-analysis} and Section~\ref{sec-conver}, and we conclude with some remarks in Section~\ref{sec: conclusion}.

	\section{Notation and Terminology} 
	\label{Sec:note-algo-asum}
	Throughout the paper, all vectors are viewed as column-vectors unless stated otherwise.
	We use $\la \cdot, \cdot\ra$ to denote the inner product, and 
	$\|\cdot\|$ to denote the standard Euclidean norm.
	We write $\mathbf{1}$ to denote the vector with all entries equal to 1, and 
	$\mathbb{I}$ to denote the identity matrix.
	The $i$-th entry of a vector $u$ is denoted by $u_i$, while it is denoted by $[u_k]_i$ 
	for a time-varying vector $u_k$. 
	Given a vector $v$, we use $\min(v) $ and $\max(v)$
	to denote the smallest and the largest entry of $v$, respectively,
	i.e., $\min(v)=\min_i v_i$ and $\max(v)=\max_i v_i$.
	A vector is said to be a stochastic vector if its entries are nonnegative and sum to 1. 	
	
	To denote the $ij$-th entry of a matrix $A$, we write $A_{ij}$, and we write $[A_k]_{ij}$ when the matrix is time-dependent.
	For any two matrices $A$ and $B$ of the same dimension,
	we write $A\le B$ to denote that $A_{ij}\le B_{ij}$ for all $i$ and $j$.
	A matrix is said to be nonnegative if all its entries are nonnegative.
	For a nonnegative matrix, we use $\min(A^+)$ to denote the smallest positive entry of $A$, i.e.,	
	$\min(A^+)=\min_{\{ij:A_{ij}>0\}} A_{ij}$.
	A nonnegative matrix is said to be row-stochastic if each row entries sum to 1, and 
	it is said to be column-stochastic if each column entries sum to 1. In particular, if 
	$A\in\mathbb{R}^{n\times n}$ is row-stochastic and 
	$B\in\mathbb{R}^{n\times n}$ is column stochastic, 
	then $A\mathbf{1}=\mathbf{1}$ and $\mathbf{1}^{\T} B=\mathbf{1}^{\T}$.
	
	Given a vector $\ba\in\re^n$ with positive entries $a_1,\ldots,a_n$, the $\ba$-weighted norm can be induced
    in the vector space $\re^p\times\cdots\times\re^p$ (consisting of $n$ copies of $\re^p$),
    as follows:
    \[\|\bx\|_\ba=\sqrt{\sum_{i=1}^n a_i\|x_i\|^2} \qquad\hbox{for $\bx=(x_1,\ldots,x_n)\in\re^p\times\cdots\times\re^p$}.\]
    When $\ba=\1$, we simply write $\|\bx\|$. We also write $\|\bx\|_{\ba^{-1}}$
    to denote the norm induced by the vector with entries $1/a_i$, i.e.,
    $\|\bx\|_{\ba^{-1}}=\sqrt{\sum_{i=1}^n\frac{\|x_i\|^2}{a_i}}$. The following relations, which can be proved by using H\"older's inequality, will be useful in our analysis: 
    \begin{subequations}
    \begin{align}\label{eq-norm-anorm}
    &\|\bx\|\le \sqrt{\tfrac{1}{\min(\ba)}}\, \|\bx\|_{\ba} \qquad\hbox{for all $\bx\in\re^p\times\cdots\times\re^p$
    and $\ba>\0$},\\
    \label{eq-stochvec-norm}
    &\|\bx\|\le \|\bx\|_{\ba^{-1}} \qquad\hbox{for all $\bx\in\re^p\times\cdots\times\re^p$
    and $\ba>0$ satisfying $\la \ba,\1\ra=1$}.
    \end{align}
    \end{subequations}
	
	We let $[n]=\{1,\ldots,n\}$ for an integer $n\ge1$.
	A directed graph $\mathbb{G}=([n],\mathcal{E})$ is specified by the edge set $\mathcal{E}\subseteq [n]\times [n]$ of ordered pairs of nodes.
	Given a directed graph $\mathbb{G}=([n],\mathcal{E})$,
	the sets $\Niout$ and $\Niin$ denote the out-neighbors and the in-neighbors of a node $i$, i.e., $\Niout=\{j\mid (i,j)\in\mathcal{E}\}$ and
	$\Niin=\{j\mid (j,i)\in\mathcal{E}\}$.

    We say that a directed graph is {\it strongly connected} if there is a directed path from any node to all other nodes in the graph. Given a directed path, the length of the path is the number of edges in the path. We use the following definitions:
	\begin{definition} [Graph Diameter] \label{def-diam}
	The diameter of a strongly connected directed graph $\mathbb{G}$, denoted by $\D(\G)$,
	is the length of the longest path in the collection of all shortest
	directed paths connecting all ordered 
	pairs of distinct nodes in $\bbG$.
	\end{definition}
	Let $\bp_{jl}$ denote a \textit{shortest directed path from node $j$ to node $l$}, where $j\ne l$. A collection $\mathcal{P}$ of directed paths in $\G$ is a shortest-path graph covering if $\bp_{jl}\in \mathcal{P}$ and $\bp_{lj}\in \mathcal{P}$ for any two
    nodes $j,~l\in [n]$, $j\ne l$. The \textit{utility of the edge} $(j,l)$ with respect to the covering $\mathcal{P}$ is the number of shortest paths in $\mathcal{P}$ that pass through the edge $(j,l)$. 
    Define $\K(\mathcal{P})$ as the maximum edge-utility in $\mathcal{P}$ taken over all edges in the graph, i.e.,
    $\displaystyle \K(\mathcal{P})=\max_{(j,l)\in\mathcal{E}}\sum_{\bp\in\mathcal{P}}\chi_{\{(j,l)\in\bp\}},$
    where $\chi_{\{(j,l)\in\bp\}}$ is the indicator function taking value 1 when $(j,l)\in\bp$ and, otherwise, taking value 0. Denote by $\mathcal{S}(\G)$  the collection of all possible shortest-path coverings of the graph $\G$, we have the following definition.

    \begin{definition} [Maximal Edge-Utility] \label{def-edgeut}
    For a strongly connected directed graph $\G=([n],\mathcal{E})$, the maximal edge-utility is the maximum value of $\K(\mathcal{P})$ taken over all possible shortest-path coverings $\mathcal{P}\in\mathcal{S}(\G)$, i.e.,
    $\K(\G)=\max_{\mathcal{P}\in\mathcal{S}(\G)}\K(\mathcal{P}).$
    \end{definition}

	\subsection{$AB$/Push-Pull Method and Assumptions}\label{ssec-algo-asum}
	We consider a system with $n$ agents, and let each agent $i\in\{1,2,\ldots,n\}$ have a local copy $x_i\in \re^p$ of the decision variable and a direction $y_i\in\mathbb{R}^p$ which is an estimate of a ``global update direction".  These variables are maintained and updated over time and at iteration $k$, they are denoted by $x_i^k$ and $y_i^k$, respectively. 
    We present a distributed  algorithm, termed $AB$/Push-Pull algorithm 
	to fairly capture independent and simultaneous
	developments of two closely related methods, namely the {\it Push-Pull} method 
	of~\cite{pshi21} and the method proposed in~\cite{xin2018linear}.
	We consider the $AB$/Push-Pull gradient method over a sequence $\{\bbG_k\}$ of
	directed graphs, where the agents communicate over a graph $\bbG_k$ at the round $k$ of updates.
	At every time $k$, the agents updates are described by two non-negative matrices $A_k$ and $B_k$ that are compliant with the connectivity structure of the graph $\bbG_k$, i.e.,
	\begin{subequations}
	\begin{align}\label{eq-agcompat}
	&[A_k]_{ij}>0\quad\hbox{for all $j\in\Nikin\cup\{i\}$},
	\qquad  &&[A_k]_{ij}=0\quad\hbox{for all $j\not\in\Nikin\cup\{i\}$},\\\label{eq-bgcompat}
	&[B_k]_{ji}>0\quad\hbox{for all $j\in\Nikout\cup\{i\}$},
	\qquad  &&[B_k]_{ji}=0\quad\hbox{for all $j\not\in\Nikout\cup\{i\}$}.
	\end{align}
	\end{subequations}
	Moreover, each matrix $A_k$ is row-stochastic and each matrix $B_k$ is column-stochastic, i.e., $A_k\1=\1$ and ${\1}^{\T} B_k={\1}^{\T}$ for all $k\ge0.$
	The method works as follows:
	at time $k$, every agent $i$ sends its vector $x_i^k$ and a scaled direction
	$[B_k]_{ji}y_i^k$ to its out-neighbors $j\in\Nikout$, while it keeps $[B_k]_{ii}y_i^k$ for its own update. 

	Upon the information exchange step, every agent $i$ updates as follows: for all $k\ge0$,
	\begin{subequations}\label{eq-method}
		\begin{align}
		&x_i^{k+1} = \sum_{j\in\Nikin}[A_k]_{ij}x_j^k  - \a y_i^k,
		\label{eq-xupdate}\\
		&y_i^{k+1} =  \sum_{j\in\Nikin}[B_k]_{ij}y_j^k + \nabla f_i(x_i^{k+1}) -\nabla f_i(x_i^k),
		\label{eq-yupdate}
		\end{align}
	\end{subequations}
	where $\a>0$ is a stepsize. 
	In this method, the agent $i$ decides on the entries of $A_k$ in the $i$th row (for the in-neighbors $j\in\Nikin$), while the value $[B_k]_{ij}$ is selected by agent $j\in \Nikin$. Each agent $i$ initializes the updates with 
	an arbitrary vector $x_i^0$ and with $y_i^0=\nabla f_i(x_i^0)$, which does not require any coordination among agents.
	The update step using the mixing matrix $A_k$ is viewed as a {\it pull-step}, while the step utilizing the matrix $B_k$ is viewed as a {\it push-step} as it is reminiscent of the push-sum consensus method. 
	
	When the matrices $A_k$ and $B_k$ are compatible with the underlying graph $\bbG_k$ (see~\eqref{eq-agcompat} and~\eqref{eq-bgcompat}), we can re-write the method~\eqref{eq-method} as follows: 
	for all $i\in[n]$ and all $k\ge0$,
	\begin{subequations}\label{eq-met}
		\begin{align}
		&x_i^{k+1} =  \sum_{j=1}^n[A_k]_{ij}x_j^k  - \a y_i^k,
		\label{eq-x}\\
		&y_i^{k+1} =  \sum_{j=1}^n [B_k]_{ij}y_j^k + \nabla f_i(x_i^{k+1}) -\nabla f_i(x_i^k),
		\label{eq-y}\\
		&\hbox{where $x_i^0\in\mathbb{R}^p$ is arbitrary and $y_i^0=\nabla f_i(x_i^0)$}.
		\label{eq-init}
		\end{align}
	\end{subequations}
	
	\noindent We analyze the convergence properties of the method under the following assumptions:

	\begin{assumption}[Strongly Connected Graphs] \label{asm-graphs}
		For each $k$, the directed graph $\bbG_k=([n],\mathcal{E}_k)$ 
		is  strongly connected.
	\end{assumption}
	
	\begin{assumption} [Graph Compatible $A_k$] \label{asm-amatrices}
		For each $k$, the matrix $A_k$ is row-stochastic 
		and compatible with the graph $\bbG_k$
		in the sense of relation~\eqref{eq-agcompat}. Moreover,
		there exists a scalar $a>0$ such that 
		$\min(A_k^+)\ge a$ for all $k\ge0$.
	\end{assumption}

	\begin{assumption}[Graph Compatible $B_k$] \label{asm-bmatrices}
		For each $k$, the matrix $B_k$ is column-stochastic 
		and compatible with the graph $\bbG_k$
		in the sense of relation~\eqref{eq-bgcompat}. Moreover,
		there exists a scalar $b>0$ such that 
		$\min(B_k^+)\ge b$ for all $k\ge0$.
	\end{assumption}
	
	\begin{assumption}
	[Lipschitz gradient] \label{asm-functions}
	Each $f_i$ is continuously differentiable and has $L$-Lipschitz continuous gradients, i.e., for some $L>0$,
	\begin{equation*}
	\|\nabla f_i(x)-\nabla f_i(u)\|\le L \|x-u\|,
	\qquad\hbox{for all $x,u\in \re^p$}.
	\end{equation*}
	\end{assumption}
	
	\begin{assumption}
	[Strong convexity]\label{asm-strconv}
	The average-sum function $f=\frac{1}{n}\sum_{i=1}^n f_i$ is $\mu$-strongly convex, i.e., for some $\mu>0$, 
	\[\la \nabla f(x)-\nabla f(u),x-u\ra \ge \mu\|x-u\|^2\qquad\hbox{for all $x,u\in \re^p$}.\]
	\end{assumption}

	\section{Basic Results}\label{sec-basic}
	\subsection{Linear Combinations and Graphs}\label{ssec-lincomb-graphs}
	
	Certain contractive properties of the iterates produced by the method 
	are inherited from the use of the mixing terms $\sum_{j=1}^n[A_k]_{ij}x_j^k$
	and $\sum_{j=1}^n[B_k]_{ij}y_j^k$, and the fact that the matrices $A_k$ and $B_k$ are compliant with a directed strongly connected graph $\bbG_k$. The following results will help us capture these contractive properties.
	
	For a collection $\{u_i, \, i\in[n]\}\subset\re^p$ of vectors and a collection $\{\g_i,\, i\in[n]\}\subset\re$ of scalars, we have the following relations (see Lemma~1 and Corollary~1 of~\cite{nguyen2022distributed}):
	\begin{equation}\label{eq-normofsum}
	\left\|\sum_{i=1}^n \g_i u_i\right\|^2
	=\left(\sum_{j=1}^n \g_j\right)\sum_{i=1}^n  \g_i \|u_i\|^2
	-\frac{1}{2}\sum_{i=1}^n \sum_{j=1}^n \g_i \g_j \|u_i-u_j\|^2.
	\end{equation}
	Moreover, if $\sum_{i=1}^n \g_i=1$ holds, then we have
	\begin{subequations}
		\begin{align}
		\frac{1}{2}\sum_{i=1}^n \sum_{j=1}^n \g_i \g_j \|u_i-u_j\|^2 =\sum_{i=1}^n \g_i \left\|u_i - \left(\sum_{\ell=1}^n \g_\ell u_\ell\right)\right \|^2&, \label{eq-aver-disp}\\
		\left\|\sum_{i=1}^n \g_i u_i - u \right\|^2 = \sum_{i=1}^n \g_i \|u_i-u\|^2 -\sum_{i=1}^n \g_i \left\|u_i - \left(\sum_{\ell=1}^n \g_\ell u_\ell\right)\right \|^2&, ~~\hbox{for all }u\in \re^p. \label{eq-gen-aver}
		\end{align}
	\end{subequations}

	We also make use of the following result.
	\begin{lemma}[\cite{nguyen2022distributed}, Lemma 2]\label{lem-xdiverse}
	Let $\bbG=([n],{\cal E})$ be a strongly connected directed graph, where a vector $x_i$ is associated with node $i$ for all $i\in[n]$. We then have
	\[\sum_{(j,\ell)\in\mathcal{E}} \|x_j - x_\ell \|^2
	\ge \frac{1}{\mathsf{D}(\bbG)\mathsf{K}(\bbG)}\sum_{j=1}^n\sum_{\ell=j+1}^n\|x_j - x_\ell\|^2,\]
	where $\mathsf{D}(\bbG)$ is the diameter of the graph $\bbG$
	and $\mathsf{K}(\bbG)$ is the maximal edge-utility in the graph 
	(see Definitions~\ref{def-diam} and~\ref{def-edgeut}). 
	\end{lemma}
	
	
	\subsection{Implications of Stochastic Nature of Matrices $A_k$ and $B_k$}\label{sec-stochmat}
	The column stochastic property of the matrices $B_k$ ensures that the sum of the $y$-iterates of the method~\eqref{eq-met}, at any time $k$,  is equal to the sum of the gradients $\nabla f_i(x_i^k)$, as seen in the following lemma. 
	\begin{lemma}
	\label{lem-sumgrad}
	Consider the method in~\eqref{eq-met}, and assume that each $B_k$ is column-stochastic.
	Then, we have
	$\sum_{i=1}^n y_i^k=\sum_{i=1}^n \nabla f_i(x_i^k)$ for all $k\ge0$.
	\end{lemma}
	
	\begin{proof}
	The proof is by the mathematical induction on $k$. 
	\end{proof}
	
	
	\begin{lemma}[\cite{nguyen2022distributed}, Lemma 3]
	\label{lem-amatrices}
	Let Assumption~\ref{asm-graphs} hold, and let $\{A_k\}$ be a matrix sequence
	satisfying Assumption~\ref{asm-amatrices}.
	Then, there exists a sequence $\{\phi_k\}$ of stochastic vectors such that
	\begin{align}\label{eq-phik}
	    \phi_{k+1}^{\T}A_s=\phi_k^{\T}\qquad\hbox{for all $k\ge0$},
	\end{align}
	where the entries of each $\phi_k$ are positive and have a uniform lower bound, i.e.,	
	\[ [\phi_k]_i\ge \frac{a^n}{n}\qquad\hbox{ for all $i\in[n]$},\]
	with $a\in(0,1)$ being the lower bound on the positive entries of the matrices $A_k$.
	\end{lemma}
	
	
	For the matrices $B_k$,
	we define the stochastic vector sequence 
	$\{\pi_k\}$ as follows:
	\begin{equation}\label{eq-pik}
	\pi_{k+1}=B_k\pi_k,\qquad\hbox{initialized with }\ \pi_0=\frac{1}{n}\1.
	\end{equation}
	We examine the sequence $\{\pi_k\}$ in the following lemma.
	\begin{lemma}
	\label{lem-bmatrices}
	Let Assumption~\ref{asm-graphs} hold and let the matrix sequence $\{B_k\}$
	satisfy Assumption~\ref{asm-bmatrices}.
	Then, the vectors $\pi_k$ generated by~\eqref{eq-pik} are stochastic vectors such that
	\[ [\pi_k]_i\ge \frac{b^n}{n}\qquad\hbox{ for all $i\in[n]$ and $k\ge0$},\]
	where $b\in(0,1)$ is the lower bound on the positive entries of the matrices 
	$B_k$.
	\end{lemma}
	\begin{proof}
	We prove  that each $\pi_k$ is stochastic by using the mathematical induction on~$k$.
	For $k=0$, the vector $\pi_0=\frac{1}{n}\1$ is stochastic. 
	Suppose now the vector $\pi_k$ is stochastic. 
	Choose any index $i\in[n]$ and consider the entry $[\pi_{k+1}]_i$. By the definition of $\pi_{k+1}$ in~\eqref{eq-pik}, since the entries in $B_k$ and $\pi_k$ are nonnegative, we
	have 
	$[\pi_{k+1}]_i =\sum_{j=1}^n [B_k]_{ij}[\pi_k]_j\ge0.$
	Furthermore, by summing the entries of $\pi_{k+1}$, and using the 
	facts that $B_k$ is column stochastic and 
	$\pi_k$ is a stochastic vector,
	we obtain
	$\1^{\T}\pi_{k+1}=\1^{\T}B_k\pi_k=\1^T\pi_k=1.$
	Thus, $\pi_{k+1}$ is a stochastic vector.
        
    To prove the lower bound result, we consider separately the case for $k=0,\ldots, n-1$ and the case $k\ge n$. The bound is obviously valid for $k=0$, since $\pi_0=\frac{1}{n}\1$.
    Let $k$ be such that  $1\le k\le n-1$. By the definition of $\pi_k$, we have
    \[\pi_k=B_{k-1}\cdots B_0 \pi_0=\frac{1}{n} B_{k-1}\cdots B_0\1.\]
    Hence, it follows that
    \[[\pi_k]_i=\frac{1}{n}[B_{k-1}\cdots B_0\1]_i = \frac{1}{n}\sum_{j=1}^n [B_{k-1}\cdots B_0]_{ij} \ge \frac{1}{n} [B_{k-1}\cdots B_0]_{ii}\ge\frac{b^k}{n},\]
    where the last inequality follows from $[B_{k-1}\cdots B_0]_{ii}\ge b^k$, which is valid since all matrices $B_k$ have positive diagonals with diagonal entries larger than or equal to $b$ (see Assumption~\ref{asm-bmatrices}). Since $k<n$, it follows that
    \[[\pi_k]_i\ge \frac{b^k}{n}>\frac{b^n}{n}\qquad\hbox{for all $k=1,\ldots,n-1$}.\]
        
    Now, consider the case $k\ge n$. Using the definition of $\pi_k$, we obtain
    \[\pi_k=B_{k-1}\cdots B_{k-n} \pi_{k-n}.\] 
    We note that the matrix $[B_{k-1}\cdots B_{k-n}]$ has all entries positive as it represents directed paths among the nodes in the composition of the strongly connected graphs 
    $\bbG_{k-1},\ldots,\bbG_{k-n}$. Moreover, every entry of
    $[B_{k-1}\cdots B_{k-n}]$ is at least as large as $b^n$, i.e.,
    \[[B_{k-1}\cdots B_{k-n}]_{ij}\ge b^n\qquad\hbox{for all }i,j\in[n],\]
    which follows by Assumption~\ref{asm-bmatrices} ensuring that each $B_t$ has positive entries on links in the graph $\bbG_t$, which are at least large as $b$. Hence, it follows that
    \[[\pi_k]_i=\sum_{j=1}^n[B_{k-1}\cdots B_{k-n}]_{ij} [\pi_{k-n}]_j
    \ge \sum_{j=1}^n b^n [\pi_{k-n}]_j=b^n>\frac{b^n}{n},\] 
	where the last equality holds since $\pi_{s}$ is a stochastic vector for all $s$.
	\end{proof}

	\subsection{Contractive Property of Gradient Method}\label{ss-grad}
	
	\begin{lemma}\label{lem-contraction} 
	For a $\mu$-strongly convex function $F$ with $L$-Lipschitz continuous gradients, 
	at the point $x^*=\argmin_{x} F(x)$, we have
	\[\|x-x^*-\a \nabla F(x)\|\le q(\a) \|x-x^*\|\qquad\hbox{for all $x$ and for all $\a$ with $0<\a<2L^{-1}$},\]
	where $q(\a)=\max\{|1-\a\mu|,|1-\a L|\}<1$. 
	\end{lemma}
	The proof of Lemma~\ref{lem-contraction} can be found within the proof of 
	Theorem 3 of Chapter 1 in~\cite{Polyak} for a twice continuously differentiable function. 
	The result has been extended in~\cite{Qu2017} (see Lemma 10 therein) to a more general case of a differentiable function.
	
	\section{Convergence Analysis}\label{sec-analysis}
	In this section, we specify
	and analyze the behavior of three quantities that are critical components of the convergence proof of the method: the distance of a suitably defined weighted average $\hat x^k$ from the solution $x^*$ of problem~\eqref{eq-problem}, a weighted dispersion of the iterates $x_i^k$ from the weighted average $\hat x^k$, and a weighted dispersion of the agents' directions $y_i^k$ from the sum $\sum_{i=1}^n y_i^k$.

	\subsection{Weighted Averages of Agents' $x$-variables}\label{ssec-waver}
	We define $\hat{x}^k$ to be the $\phi_k$-weighted averages of the iterates $x_i^k$ produced by the $AB$/Push-Pull method~\eqref{eq-met}, i.e.,
	\begin{equation}\label{eq-waverx}
	\hat x^k = \sum_{i=1}^n [\phi_k]_i x_i^k\qquad\hbox{for all }k\ge0,
	\end{equation}
	where $\{\phi_k\}$ is the sequence of stochastic vectors satisfying $\phi_{k+1}^{\T}A_k=\phi_k^{\T}$ (see Lemma~\ref{lem-amatrices}). In the next proposition, we establish a recursion relation for $\hat x^k$,
    and a relation for their distance from the optimal solution $x^*$ of problem~\eqref{eq-problem}. 
	
	\begin{proposition}
	\label{prop-waverx}
	Let Assumptions~\ref{asm-amatrices}-\ref{asm-strconv} hold. 
	Then, the following statements are valid:
	\begin{itemize}
	\item[(a)] The weighted average sequence $\{\hat x^k\}$ defined in~\eqref{eq-waverx} satisfies,
	\begin{equation}\label{eq-hatx}
	\hat x^{k+1}=\hat x^k - \a \sum_{i=1}^n [\phi_{k+1}]_i y_i^k \qquad\hbox{for all $k\ge0$}.
	\end{equation}
    \item[(b)] 
	Let the stepsize $\a$ in method~\eqref{eq-met} be such that
	$0<\a<\tfrac{2}{nL}$, where $L$ is the gradient Lipschitz constant 
	from Assumption~\ref{asm-functions}.
	Then, we have for all $k\ge0$, 
	\begin{align*}
	\|\hat x^{k+1}-x^*\| \le q_k(\a) \|\hat x^k - x^*\| + \a L\sqrt{\dfrac{n}{\min (\phi_{k})}}D(\bx^k,\phi_k)+\a S(\by^k,\pi_k),
	\end{align*}
	where $q_k(\a) = \max\{|1-\a n \min(\pi_k)\mu|,|1-\a n \min(\pi_k)L|\} <1$.
	\end{itemize}
	\end{proposition}
	\begin{proof}
	(a)~By the definition of $\hat x^{k+1}$ and
	the $x$-update relation given in~\eqref{eq-x}, we have 
	\[\hat x^{k+1}=\sum_{i=1}^n [\phi_{k+1}]_i x_i^{k+1}
	=\sum_{i=1}^n [\phi_{k+1}]_i \sum_{j=1}^n [A_k]_{ij}x_j^k - \a \sum_{i=1}^n [\phi_{k+1}]_i y_i^k.\]
	For the double-sum term, it follows that
	\[\sum_{i=1}^n [\phi_{k+1}]_i \sum_{j=1}^n [A_k]_{ij}x_j^k
	=\sum_{j=1}^n \left(\sum_{i=1}^n [\phi_{k+1}]_i [A_k]_{ij}\right) x_j^k
	=\sum_{j=1}^n [\phi_k]_j x_j^k=\hat x^k,\]
	where the second equality follows by $\phi_{k+1}^{\T}A_k=\phi_k^{\T}$ (see
	Lemma~\ref{lem-amatrices}), and the last equality uses the definition of $\hat x^k$, thus, establishes the desired relation in part~(a).
	\noindent
	(b)~Under Assumption~\ref{asm-strconv}, the unique minimizer $x^*$ of $f(x)$ over $x\in\re^p$ exists. 
	By subtracting the optimal point $x^*$ from both sides of the relation in part (a) 
	(see~\eqref{eq-hatx}), and by
	adding and subtracting $\sum_{i=1}^n [\phi_{k+1}]_i \a n [\pi_k]_i \nabla f(\hat x^k)$, we obtain
	\[\hat x^{k+1}-x^*=\hat x^k - x^* - \sum_{i=1}^n [\phi_{k+1}]_i\a n [\pi_k]_i\nabla f(\hat x^k) 
	+\a \sum_{i=1}^n [\phi_{k+1}]_i\left(n [\pi_k]_i \nabla f(\hat x^k) -  y_i^k\right).\]
	Therefore, by the convexity of the norm and the fact that $\phi_{k+1}$ is stochastic, we have
	\[\|\hat x^{k+1}-x^*\|\le \sum_{i=1}^n [\phi_{k+1}]_i\|\hat x^k - x^* - \a n [\pi_k]_i\nabla f(\hat x^k) \|
	+\a \sum_{i=1}^n[\phi_{k+1}]_i \|y_i^k-n [\pi_k]_i \nabla f(\hat x^k)\|.\]
	By Assumption~\ref{asm-functions} and Assumption~\ref{asm-strconv}, the function $f$ is $\mu$-strongly convex and 
	has $L$-Lipschitz continuous gradients. 
	Thus, 
	for a stepsize $\a$ satisfying $\a\in  (0,\tfrac{2}{n[\pi_k]_iL})$, for all $i\in [n]$, by Lemma~\ref{lem-contraction}
	it follows that 
	\[\|\hat x^k - x^* - \a n [\pi_k]_i\nabla f(\hat x^k) \| \le q_{i,k}(\a) \|\hat x^k - x^*\|,\]
	with $q_{i,k}(\a)=\max\{|1-\a n [\pi_k]_i\mu|,|1-\a n [\pi_k]_i L|\}$. 

	Let $q_k(\a) = \max\{|1-\a n \min(\pi_k)\mu|,|1-\a n \min(\pi_k)L|\}<1$, using the preceding relation with the stochasticity of $\phi_{k+1}$ yields
	\begin{align*}
    \sum_{i=1}^n[\phi_{k+1}]_i\|\hat x^k - x^*-\a n [\pi_k]_i\nabla f(\hat x^k)\| \le \sum_{i=1}^n[\phi_{k+1}]_i q_{i,k}(\a) \|\hat x^k - x^*\| \le q_k(\a) \|\hat x^k - x^*\|.
    \end{align*}
    Therefore,
    \begin{align}\label{eq-hatx1}
    \|\hat x^{k+1}-x^*\| \le q_k(\a) \|\hat x^k - x^*\| + \a \sum_{i=1}^n[\phi_{k+1}]_i \|y_i^k-n [\pi_k]_i \nabla f(\hat x^k)\|.
    \end{align}
	Since $\max(\phi_{k+1})\le 1$, to estimate the last term in~\eqref{eq-hatx1}, we factor out $[\pi_k]_i$ as follows
     \begin{align*}
    \sum_{i=1}^n[\phi_{k+1}]_i\|y_i^k\!-\!n [\pi_k]_i \nabla f(\hat x^k)\| \le  \!\sum_{i=1}^n\|y_i^k\!-\!n [\pi_k]_i \nabla f(\hat x^k)\| = \!\sum_{i=1}^n[\pi_k]_i \left\|\dfrac{y_i^k}{[\pi_k]_i}\!-\!n \nabla f(\hat x^k)\right\|.
    \end{align*}
    We add and subtract $\sum_{\ell=1}^n y_\ell^k$, and use the triangle inequality for the norm to obtain
    \begin{align*}
    &\sum_{i=1}^n[\pi_k]_i \left\|\dfrac{y_i^k}{[\pi_k]_i}-n \nabla f(\hat x^k)\right\|  \le \sum_{i=1}^n[\pi_k]_i \left\|\dfrac{y_i^k}{[\pi_k]_i}-\sum_{\ell=1}^n y_\ell^k\right\| +\sum_{i=1}^n[\pi_k]_i\left\|\sum_{\ell=1}^n y_\ell^k-n\nabla f(\hat x^k)\right\|\cr
    &\!\le\! \sqrt{\sum_{i=1}^n[\pi_k]_i \left\|\dfrac{y_i^k}{[\pi_k]_i}\!-\!\sum_{\ell=1}^n y_\ell^k\right\|^2} \!+\left\|\sum_{\ell=1}^n y_\ell^k\!-\!n\nabla f(\hat x^k)\right\|
    \!\le S(\by^k,\pi_k) \!+\left\|\sum_{\ell=1}^n y_\ell^k\!-\!n\nabla f(\hat x^k)\right\|.
    \end{align*}
    Combining the two preceding relations yields
    \begin{align} \label{eq-hatx2}
     \sum_{i=1}^n[\phi_{k+1}]_i\|y_i^k-n [\pi_k]_i \nabla f(\hat x^k)\| \le S(\by^k,\pi_k) +\left\|\sum_{\ell=1}^n y_\ell^k-n\nabla f(\hat x^k)\right\|.
    \end{align}
    By Lemma~\ref{lem-sumgrad},
	$\sum_{\ell=1}^n y_\ell^k=\sum_{\ell=1}^n \nabla f_\ell(x_\ell^k)$,
	hence, in view of $f=\frac{1}{n}\sum_{\ell=1}^n f_\ell$, we have
	\[\left\|\sum_{\ell=1}^n y_\ell^k-n\nabla f(\hat x^k)\right\|
	= \left\|\sum_{\ell=1}^n \left(\nabla f_\ell(x_\ell^k) - \nabla f_\ell(\hat x^k)\right)\right\|
	\le \sum_{\ell=1}^n \|\nabla f_\ell(x_\ell^k) -  \nabla f_\ell(\hat x^k)\|.\]
	Since each $f_i$ has L-Lipschitz continuous gradients (by Assumption~\ref{asm-functions}), it follows that
	\begin{align}\label{eq-hatx3}
	\left\|\sum_{\ell=1}^n y_\ell^k-n\nabla f(\hat x^k)\right\| \le L\sum_{\ell=1}^n \|x_\ell^k -  \hat x^k\| \le  L\sqrt{\dfrac{n}{\min (\phi_{k})}}D(\bx^k,\phi_k).
	\end{align}
	Substituting \eqref{eq-hatx2} and \eqref{eq-hatx3} in relation~\eqref{eq-hatx1} gives the desired relation in part (b).
	\end{proof}

The condition $q_k(\a)<1$ of Proposition~\ref{prop-waverx}(b) holds for example when  $\a\in(0,\tfrac{2}{nL})$.

	\subsection{Weighted Dispersion of Agents' $x$-variables}\label{ssec-wdisp}
	In this section, we define and analyze the behavior of a $\phi_k$-weighted dispersion of the iterates $x_i^k$, $i\in[n]$, of the method~\eqref{eq-met} from their weighted average $\hat x^k$, i.e., 
	\begin{equation}\label{eq-def-disp}
	D(\bx^k,\phi_k)=\sqrt{\sum_{j=1}^n [\phi_k]_j \|x_j^k -\hat x^k\|^2} \qquad\hbox{for all }k\ge0,
	\end{equation}
	where the stochastic vectors $\phi_k$ satisfy $\phi_{k+1}^{\T}A_k=\phi_k^{\T}$
	and $\bx^k=(x_1^k,\ldots,x_n^k)$.
	
	The dispersion $D(\bx^k,\phi_k)$ can be interpreted as the $\phi_k$-weighted
	norm of the difference between $\bx^k$ and the vector $\hat \bx^k=(\hat x^k,\ldots,\hat x^k)$ consisting of $n$-copies of 
	$\hat x^k$, i.e.,
	\begin{equation}\label{eq-disp-norm}
	D(\bx^k,\phi_k)=\|\bx^k - \hat \bx^k\|_{\phi_k}.\end{equation}
	
	\noindent Using the definition of $x_i^{k+1}$ in~\eqref{eq-x}, we can write
	\begin{equation}\label{eq-xandz}
	x_i^{k+1}=z_i^k -\a y_i^k,\qquad z_i^k=\sum_{j=1}^n [A_k]_{ij} x_j^k,
	\qquad\hbox{for all $i\in[n]$ and all $k\ge 0$}.\end{equation}
	Define $\bx^{k+1}=(x_1^{k+1},\ldots, x_n^{k+1})$
	and, similarly, define $\bz^k=(z_1^k,\ldots,z_n^k)$ and $\by^k=(y_1^k,\ldots,y_n^k)$.
	Then, we can write the preceding relations compactly as follows
	\begin{equation}\label{eq-xcomp}
	\bx^{k+1} =\bz^k - \a \by^k\qquad\hbox{for all $k\ge 0$}.\end{equation}
	
	We start our analysis by recalling the next lemma:
	
	\begin{lemma}[\cite{nguyen2022distributed}, Lemma 6]\label{lem-basic-xcontract}
	Let $\bbG=([n],{\cal E})$ be a strongly connected directed graph,  
	and let $A$ be an row-stochastic matrix that is compatible with the graph and has positive diagonal entries, i.e., $A_{ij}>0$ when $j=i$ and $(j,i)\in {\cal E}$, and $A_{ij}=0$ otherwise. 
	Also, let $\phi$  be a stochastic vector and let $\pi$ be a nonnegative vector such that $\pi^{\T}A=\phi^{\T}$.
	
	Let $x_1,\ldots,x_n\in\re^p$
	be a given collection of vectors, and consider the vectors $z_i=\sum_{j=1}^n A_{ij} x_j$ for all $i\in[n]$. Then, we have
	\[\sum_{i=1}^n \pi_i\|z_i- u\|^2 \le 
	\sum_{j=1}^n \phi_j\|x_j - u\|^2
	-  
	\frac{\min(\pi) \left(\min(A^+)\right)^2}{\max^2(\phi)\mathsf{D}(\bbG)\mathsf{K}(\bbG)}
	\sum_{j=1}^n \phi_j \|x_j - \hat x_\phi \|^2 ~~ \hbox{for all }u\in\re^p,\]
	where $\mathsf{D}(\bbG)$ and $\mathsf{K}(\bbG)$ are the diameter and the maximal edge-utility of $\bbG$,
	respectively.
	\end{lemma}
	
	The relation for the dispersion $D(\bx^k,\phi_k)$ is given in the following proposition.
	
	\begin{proposition}\label{prop-xcontract}
    Let Assumption~\ref{asm-graphs} and Assumption~\ref{asm-amatrices} hold.
	We have for all $k\ge0$,
	\[D(\bx^{k+1},\phi_{k+1})\le 
	c_k D(\bx^k,\phi_k)
	+ \a \sqrt{\sum_{i=1}^n [\phi_{k+1}]_i \left\|y_i^k - \sum_{j=1}^n[\phi_{k+1}]_j y_j^k\right\|^2},\]
	where the scalar $c_k\in(0,1)$ is given by
	$c_k=\sqrt{1- \frac{\min(\phi_{k+1})\, a^2}
	{\max^2(\phi_k)\,\mathsf{D}(\bbG_k)\mathsf{K}(\bbG_k)} }$ and
	$\phi_k$ are the stochastic vectors from Lemma~\ref{lem-amatrices} 
	\end{proposition}
	\begin{proof}
	We define ${\bf v}^k=(\sum_{j=1}^n[\phi_{k+1}]_j y_j^k,\ldots,\sum_{j=1}^n[\phi_{k+1}]_j y_j^k)$,
	for which we can write the relation for the weighted averages in Proposition~\ref{prop-waverx}(a) in compact form, as follows
	\[\hat \bx^{k+1}=\hat \bx^{k} -\a {\bf v}^k.\]
	Upon subtracting the preceding relation and 
	the compact representation of $x$-iterate process in~\eqref{eq-xcomp}, we obtain
	\[\bx^{k+1} - {\bf \hat x}^{k+1} = \bz^k - \hat \bx^{k} - \alpha \left(\by^k - {\bf v}^k\right).\]
	Taking the $\phi_{k+1}$-norm on both sides of the preceding relation and using the triangle inequality
	and the positive scaling property of a norm, we obtain
	\[\|\bx^{k+1}-\hat \bx^{k}\|_{\phi_{k+1}}
	=\|\bz^k - \hat \bx^{k} -\alpha \left(\by^k- {\bf v}^k\right)\|_{\phi_{k+1}}
	\le \|\bz^k - \hat \bx^{k}\|_{\phi_{k+1}} + \alpha \|\by^k- {\bf v}^k\|_{\phi_{k+1}}.\]
	The left hand side of the preceding relation corresponds to 
	the dispersion $D(\bx^{k+1},\phi_{k+1})$ (see~\eqref{eq-disp-norm}). 
	The terms on the right hand side we write explicitly in terms of the vector components with $z_i^k=\sum_{j=1}^n [A_k]_{ij} x_j^k$ (see~\eqref{eq-xandz}),
	and obtain 
	\begin{align}\label{eq-11}
	D(\bx^{k+1}\!,\phi_{k+1})
	\le \!\sqrt{\sum_{i=1}^n [\phi_{k+1}]_i \|z_i^k -\hat x^k\|^2} 
	+\a 
	\sqrt{\sum_{i=1}^n [\phi_{k+1}]_i \left\|y_i^k - \!\sum_{j=1}^n[\phi_{k+1}]_j y_j^k\right\|^2}.
	\end{align}
	
	Next, we note that the vectors $z_i^k$, $i\in[n]$, 
	satisfy Lemma~\ref{lem-basic-xcontract}, with $A=A_k$, and 
	$x_i=x_i^k$ for all $i\in[n]$. Moreover, since we have $\phi_{k+1}^{\T}A_k=\phi_k^{\T}$
	by Lemma~\ref{lem-amatrices}, Lemma~\ref{lem-basic-xcontract} applies with 
	$\pi=\phi_{k+1}$, $\phi=\phi_k$ and $\hat x_\phi=\hat x^k$, which yields
	\begin{equation}\label{eq-zxhat}
	\sum_{i=1}^n [\phi_{k+1}]_i\|z_i^k - \hat x^k\|^2 
	\le \left(1- \frac{\min(\phi_{k+1})\, a^2}
	{\max^2(\phi_k)\,\mathsf{D}(\bbG_k)\mathsf{K}(\bbG_k)}\right)
	\sum_{j=1}^n [\phi_k]_j \|x_j^k - \hat x^k\|^2, 
	\end{equation}
	where we use $\min(A_k^+)\ge a$ (see Assumption~\ref{asm-amatrices}).
	Therefore,
	\begin{equation}\label{eq-sqrt}
	\sqrt{\sum_{i=1}^n [\phi_{k+1}]_i \|z_i^k -\hat x^k\|^2}
	\le c_k
	\sqrt{\sum_{j=1}^n [\phi_k]_j \|x_j^k - \hat x^k\|^2},\end{equation}
	where 
	$c_k=\sqrt{1- \frac{\min(\phi_{k+1} )\,a^2}
	{\max^2(\phi_k)\,\mathsf{D}(\bbG_k)\mathsf{K}(\bbG_k)}}.$
	
	By recognizing the term on the right hand side of~\eqref{eq-sqrt} 
	corresponds to $D(\bx^k,\phi_k)$,
	and by combining estimate~\eqref{eq-sqrt} with~\eqref{eq-11}, we obtain the desired relation.
	\end{proof}
	
	We conclude this section with a result establishing an estimate for $\|\bx^{k+1}-\bx^k\|$
	and $\left\|\sum_{i=1}^n y_i^k\right\|$,
	which will be soon used in the analysis of the behavior of the $y$-iterates.
	
	\begin{lemma}\label{lem-xdiff}
    Let Assumption~\ref{asm-graphs} and Assumption~\ref{asm-amatrices} hold.
	Then, for all $k\ge0$,
    we have 
	\begin{align}\label{eq-xdiff}
	\left\|\bx^{k+1}-\bx^k\right\|
	\le \left(c_k\sqrt{\frac{1}{\min(\phi_{k+1})}}+\sqrt{\frac{1}{\min(\phi_k)}} \right) D(\bx^k,\phi_k)+ \alpha \|\by^k\|.
	\end{align}
	Additionally, if Assumption~\ref{asm-bmatrices} and Assumption~\ref{asm-functions} hold.
	Then we have for all $k\ge0$,
	\[\left\|\sum_{i=1}^n y_i^k\right\|
	\le L\sqrt{\frac{n}{\min(\phi_k)}} \left(\|\hat x^k - x^*\| + D(\bx^k,\phi_k)\right).\]
	\end{lemma}
	\begin{proof}
	Adding and subtracting $\hat \bx^{k}=(\hat x^k,\ldots,\hat x^k)$, we have
	\[\left\|\bx^{k+1}-\bx^k\right\|=\left\|\bx^{k+1} - {\hat \bx}^k +{\hat \bx}^k-\bx^k\right\|
	\le \left\|\bz^k-{\hat \bx}^k\right\|+\left\|\bx^k - {\hat \bx}^k\right\| + \alpha \|\by^k\|,\]
	where the last inequality follows from the compact representation of $x$-iterate process (see~\eqref{eq-xcomp}) and the triangle inequality.
	By the relation for norms in~\eqref{eq-norm-anorm},
	it follows that 
	\begin{align*}
	\left\|\bx^{k+1}-\bx^k\right\|
	\le \sqrt{\frac{1}{\min(\phi_{k+1})}} \, \left\|\bz^k-{\hat \bx}^k\right\|_{\phi_{k+1}}
	+\sqrt{\frac{1}{\min(\phi_k)}} \left\|\bx^k - {\hat \bx}^k\right\|_{\phi_k}+ \alpha \|\by^k\|.\end{align*}
	We notice that by the relation in \eqref{eq-zxhat}, we have
	$\displaystyle \left\|\bz^k-{\hat \bx}^k\right\|_{\phi_{k+1}} \le c_k \left\|\bx^k - {\hat \bx}^k\right\|_{\phi_k}$. Thus, we obtain the first relation in \eqref{eq-xdiff} upon using the definition of $D(\bx^k,\phi_k)$ in~\eqref{eq-def-disp}.
	
	Now, we consider $\left\|\sum_{i=1}^n y_i^k\right\|$.
	By Lemma~\ref{lem-sumgrad}, we have 
	\[\left\| \sum_{i=1}^n y_i^k\right\|=\left\|\sum_{i=1}^n \nabla f_i(x_i^k)\right\|
	=\left\|\sum_{i=1}^n\left(\nabla f_i(x_i^k) -\nabla f_i(x^*)\right) \right\|,\]
	where we use the fact that $\sum_{i=1}^n\nabla f_i(x^*)=0$, which holds since $x^*$ 
	is the solution to problem~\eqref{eq-problem}.
	Therefore, by using the assumption that each $f_i$ has Lipschitz continuous gradients 
	with a Lipschitz constant $L>0$, we obtain
	\[\left\| \sum_{i=1}^n y_i^k\right\|
	\le \sum_{i=1}^n\left\| \nabla f_i(x_i^k)-\nabla f_i(x^*)\right\|
	\le L\sum_{i=1}^n\| x_i^k -x^*\|= L\sqrt{n}\|\bx^k -\bx^*\|.\]
	Using the relation for norms in~\eqref{eq-norm-anorm}, we further obtain
	\begin{align}\label{eq-sumy1}
	\left\| \sum_{i=1}^n y_i^k\right\|
	\le L\sqrt{\frac{n}{\min(\phi_k)}}\|\bx^k -\bx^*\|_{\phi_k}.
	\end{align}
	Applying the relation~\eqref{eq-gen-aver} with $u_i=x_i^k$, $\g_i=[\phi_k]_i$ for all $i$, and $u=x^*$ yields
	\[\sum_{i=1}^n [\phi_k]_i\|x_i^k-x^*\|^2
	=\|\hat x^k-x^*\|^2 + \sum_{i=1}^n [\phi_k]_i\|x_i^k-\hat x^k\|^2,\]
	where $\hat x^k=\sum_{\ell=1}^n[\phi_k]_\ell x_\ell^k$.
	Hence,
	\[\left\|\bx^k - \bx^*\right\|_{\phi_k}
	\!=\!\sqrt{\|\hat x^k-x^*\|^2 + \sum_{i=1}^n [\phi_k]_i\|x_i^k-\hat x^k\|^2}
	\le \|\hat x^k-x^*\|+\sqrt{\sum_{i=1}^n [\phi_k]_i\|x_i^k-\hat x^k\|^2},\]
	where the inequality in the preceding relation follows from $\sqrt{a+b}\le \sqrt{a}+\sqrt{b}$,
	which is valid for any $a,b\ge0$.
	Therefore, using the definition of $D(\bx^k,\phi_k)$ in~\eqref{eq-def-disp},
	we have
	\begin{equation}\label{eq-xrel2}
	\left\|\bx^k - \bx^*\right\|_{\phi_k}
	\le  \|\hat x^k-x^*\|+D(\bx^k,\phi_k),
	\end{equation}
	from which the second desired relation follows by using~\eqref{eq-sumy1} and \eqref{eq-xrel2}.
	\end{proof}

	\subsection{Weighted Dispersion of Scaled Agents' $y$-variables}\label{ssec-yvar}
 	In this section, we analyze the behavior of the directions $y_i^k$ generated by the method in~\eqref{eq-met}.
	A preliminary result that establishes a basic relation corresponding to a 
	column-stochastic matrix $B$ is given in the following lemma. 
	
	\begin{lemma}\label{lem-basic-ycontract}
	Let $\bbG=([n],{\cal E})$ be a strongly connected directed graph,  and
	let $B$ be an $n\times n$ column-stochastic matrix that is compatible with 
	the graph and has positive diagonal entries, i.e.,
	$B_{ij}>0$ when $j=i$ and $(j,i)\in {\cal E}$, and $B_{ij}=0$ otherwise. 
	Also, let $\nu$  be a stochastic vector with all entries positive, i.e., $\nu_i>0$ for all $i\in[n]$,
	and let the vector $\pi$ be given by
	$\pi=B\nu$.
	Let $y_1,\ldots,y_n\in\re^p$
	be a given collection of vectors, and consider the vectors $w_i=\sum_{j=1}^n B_{ij} y_j$ for all $i\in[n]$. Then, we have
	\begin{align*}
	\sqrt{\sum_{i=1}^n\pi_i \left\|\frac{w_i}{\pi_i} - \sum_{\ell=1}^m y_\ell\right\|^2}
	\le\tau\,
	\sqrt{\sum_{i=1}^n \nu_i \left\|\frac{y_i}{\nu_i} - \sum_{\ell=1}^n y_\ell\right\|^2},
	\end{align*}	
	where the scalar $\tau\in(0,1)$ is given by
	$\tau=\sqrt{1 -   \frac{\min^2(\nu)\,\left(\min(B^+)\right)^2}
	{\max^2(\nu) \max(\pi)\, \mathsf{D}(\bbG)\mathsf{K}(\bbG)}},$
	where $\mathsf{D}(\bbG)$ and $\mathsf{K}(\bbG)$ are the diameter and the maximal edge-utility of the graph $\bbG$, respectively.	
	\end{lemma}
	\begin{proof}
	For any $i\in[n]$, by the definition of $w_i$, we have
	\[\|w_i \|^2 =\left\| \sum_{j=1}^n B_{ij}y_j \right\|^2
	=\left\|\sum_{j=1}^n B_{ij}\nu_j \frac{y_j}{\nu_j}\right\|^2.\]
	We further expand the squared norm term by using Lemma~\ref{eq-normofsum}
	with $\g_j=B_{ij}\nu_j$ and $u_j=y_j/ \nu_j$ for all $j\in[n]$.
	Hence, we obtain
	\begin{align*}
	\|w_i \|^2 
	= \left(\sum_{\ell=1}^n B_{i\ell}\nu_\ell\right) \sum_{j=1}^n B_{ij}\nu_j\left\|\frac{y_j}{\nu_j}\right\|^2 
	-\frac{1}{2} \sum_{j=1}^n \sum_{\ell=1}^n B_{ij} \nu_jB_{i\ell}\nu_\ell
	 \left\|\frac{y_j}{\nu_j} - \frac{y_\ell}{\nu_\ell} \right\|^2.
	\end{align*}
	Recalling the definition of $\pi$, i.e., $\pi=B\nu$,
	we have $\pi_i=\sum_{\ell=1}^n B_{i\ell}\nu_\ell$,
	so that we have
	
	\begin{align*}
	\|w_i \|^2 
	= \pi_i \sum_{j=1}^n B_{ij}\nu_j\left\|\frac{y_j}{\nu_j}\right\|^2 
	-\frac{1}{2} \sum_{j=1}^n \sum_{\ell=1}^n B_{ij} \nu_jB_{i\ell}\nu_\ell
	 \left\|\frac{y_j}{\nu_j} - \frac{y_\ell}{\nu_\ell} \right\|^2.
	\end{align*}
	Since the matrix $B$ is nonnegative and compatible with a strongly connected graph $\bbG$,
	and since the vector $\nu$ has all positive entries,
	it follows that the vector $\pi$ also has all entries positive. 
	By dividing with $\pi_i$ both sides of the preceding relation, 
	and then by summing over all $i$, we obtain
	\begin{align*}
	\sum_{i=1}^n\pi_i^{-1}\|w_i \|^2 
	=\sum_{i=1}^n \sum_{j=1}^n B_{ij}\nu_j\left\|\frac{y_j}{\nu_j}\right\|^2 
	- \frac{1}{2} \sum_{i=1}^n\pi_i^{-1}\sum_{j=1}^n \sum_{\ell=1}^n 
	B_{ij} \nu_jB_{i\ell}\nu_\ell
	 \left\|\frac{y_j}{\nu_j} - \frac{y_\ell}{\nu_\ell} \right\|^2.
	\end{align*}
	For the first term on the right hand side of the preceding inequality,
	we have that
	\[\sum_{i=1}^n \sum_{j=1}^n B_{ij}\nu_j\left\|\frac{y_j}{\nu_j}\right\|^2
	= \sum_{i=1}^n \sum_{j=1}^n B_{ij}\nu_j^{-1}\|y_j\|^2
	=\sum_{j=1}^n \left(\sum_{i=1}^n B_{ij}\right)\nu_j^{-1}\|y_j\|^2 
	=\sum_{j=1}^n\nu_j^{-1}\|y_j\|^2,\]
	where the last equality follows since the matrix $B$ is column-stochastic.
	Therefore, 
	\begin{equation}\label{eq-sum1}
	\sum_{i=1}^n\pi_i^{-1}\|w_i \|^2 
	=\sum_{j=1}^n \nu_j^{-1}\|y_j\|^2 
	- \frac{1}{2} \sum_{i=1}^n\pi_i^{-1}\sum_{j=1}^n \sum_{\ell=1}^n 
	B_{ij} \nu_jB_{i\ell}\nu_\ell
	 \left\|\frac{y_j}{\nu_j} - \frac{y_\ell}{\nu_\ell} \right\|^2.
	\end{equation}	
	We note that the vector $\pi$ is stochastic since $B$ is column stochastic and $\nu$ is a stochastic vector.
	Hence, 
	\[\sum_{i=1}^n\pi_i^{-1}\|w_i \|^2
	=\sum_{i=1}^n\pi_i\left\|\frac{w_i }{\pi_i}\right\|^2
	=\sum_{i=1}^n\pi_i \left\|\frac{w_i}{\pi_i} - \sum_{\ell=1}^m w_\ell\right\|^2 
	+ \left\|\sum_{\ell=1}^n w_\ell\right\|^2,\]
	where the last relation is obtained by using relation~\eqref{eq-gen-aver} with $u=0$, $u_i=w_i/\pi_i$, and 
	$\g_i=\pi_i$. Using a similar line of arguments, since $\nu$ is a stochastic vector, we obtain
	\[\sum_{j=1}^n \nu_j^{-1}\|y_j\|^2 
	=\sum_{j=1}^n\nu_j \left\|\frac{y_j}{\nu_j} - \sum_{\ell=1}^m y_\ell\right\|^2 
	+ \left\|\sum_{\ell=1}^n y_\ell\right\|^2.\]
	Since $B$ is column-stochastic, we also have $\sum_{\ell=1}^n w_\ell=\sum_{\ell=1}^n y_\ell$,
	so that by combining the  preceding two relations with~\eqref{eq-sum1},
	we have that
	\begin{equation}\label{eq-sum2}
	\!\!\sum_{i=1}^n\pi_i \left\|\frac{w_i}{\pi_i} \!-\! \sum_{\ell=1}^m y_\ell\right\|^2 
	\!\!\!\!=\!\sum_{j=1}^n\nu_j \left\|\frac{y_j}{\nu_j} \!-\!\! \sum_{\ell=1}^m y_\ell\right\|^2 
	\!\!\!\!- \frac{1}{2} \!\sum_{i=1}^n\pi_i^{-1}\!\sum_{j=1}^n \sum_{\ell=1}^n 
	B_{ij} \nu_jB_{i\ell}\nu_\ell
	 \left\|\frac{y_j}{\nu_j} \!-\! \frac{y_\ell}{\nu_\ell} \right\|^2\!\!\!.
	\end{equation}	
	Next, we estimate the second term on the right hand side of~\eqref{eq-sum2}. 
	By exchanging the order of the summation 
	so that the summation over $i$ is the last in the order,
	we obtain
	\begin{align*}
	\sum_{i=1}^n \pi_i^{-1}\sum_{j=1}^n \sum_{\ell=1}^n 
	B_{ij} \nu_jB_{i\ell}\nu_\ell
	 \left\|\frac{y_j}{\nu_j} - \frac{y_\ell}{\nu_\ell} \right\|^2
	&=\sum_{j=1}^n \sum_{\ell=1}^n  \nu_j\nu_\ell 
	\left\|\frac{y_j}{\nu_j} - \frac{y_\ell}{\nu_\ell} \right\|^2
	\left(\sum_{i=1}^n \pi_i^{-1}B_{ij}B_{i\ell}\right)\cr
	&\ge \sum_{j=1}^n \sum_{\ell\in\mathcal{N}_j^{\rm in} } 
	 \nu_j\nu_\ell \left\|\frac{y_j}{\nu_j} - \frac{y_\ell}{\nu_\ell} \right\|^2
	\left(\sum_{i=1}^n \pi_i ^{-1}B_{ij} B_{i\ell}\right).
	\end{align*}
	The graph $\bbG$ is strongly connected implying that 
	every node $j$ must have a nonempty in-neighbor set 
	$\mathcal{N}_j^{\rm in}$. 
	Moreover, by assumption we have that $B_{jj}>0$ every $j\in[n]$
	and $B_{j\ell}>0$ for  all $\ell\in\mathcal{N}_j^{\rm in} $. Therefore, it follows that
	\begin{align*}
	\sum_{i=1}^n \pi_i^{-1}B_{ij} B_{i\ell} 
	\ge \pi_j^{-1}B_{jj} B_{j\ell} 
	\ge \pi_j^{-1} \left(\min_{ij: B_{ij}>0} B_{ij}\right)^2
	\ge \left(\min_{j\in[n]}\pi_j^{-1} \right)\left(\min_{ij: B_{ij}>0} B_{ij}\right)^2. 
	\end{align*}
	Using the notation $\min(B^+)=\min_{ij: B_{ij}>0} B_{ij}$, we have
	\begin{align}\label{eq-lowerby1}
	\sum_{i=1}^n \pi_i^{-1}\sum_{j=1}^n \sum_{\ell=1}^n 
	B_{ij} \nu_jB_{i\ell}\nu_\ell
	 \left\|\frac{y_j}{\nu_j} - \frac{y_\ell}{\nu_\ell} \right\|^2
	&\ge \frac{ \left(\min(B^+)\right)^2}{\max(\pi) }
	\sum_{j=1}^n \sum_{\ell\in\mathcal{N}_j^{\rm in} }
	\nu_j\nu_\ell \left\|\frac{y_j}{\nu_j} - \frac{y_\ell}{\nu_\ell} \right\|^2\cr
	&\ge \frac{\min^2(\nu)\left(\min(B^+)\right)^2}{ \max(\pi) } 
	\sum_{(\ell,j)\in\mathcal{E}}  \left\|\frac{y_j}{\nu_j} - \frac{y_\ell}{\nu_\ell} \right\|^2\!\!.
	\end{align}
	We bound the sum $\sum_{(\ell,j)\in\mathcal{E}} \|y_j - y_\ell \|^2$
	from below by employing Lemma~\ref{lem-xdiverse}. 
	By assumption the graph $\bbG=([n],{\cal E})$ is strongly connected, by Lemma~\ref{lem-xdiverse} it follows that 
	\[\sum_{(j,\ell)\in\mathcal{E}} \left\|\frac{y_j}{\nu_j} - \frac{y_\ell}{\nu_\ell} \right\|^2
	\!\!\!\ge \frac{1}{\mathsf{D}(\bbG)\mathsf{K}(\bbG)}\sum_{j=1}^n\sum_{\ell=j+1}^n
	\!\left\|\frac{y_j}{\nu_j} - \frac{y_\ell}{\nu_\ell} \right\|^2
	\!\!\!=\! \frac{1}{\mathsf{D}(\bbG)\mathsf{K}(\bbG)}\,
	\frac{1}{2}\sum_{j=1}^n\sum_{\ell=j}^n\left\|\frac{y_j}{\nu_j} - \frac{y_\ell}{\nu_\ell} \right\|^2	\!\!\!,\]
	where $\mathsf{D}(\bbG)$ is the diameter of the graph $\bbG$, and $\mathsf{K}(\bbG)$ is the maximal edge-utility in the graph $\bbG$. 
	The preceding relation and relation~\eqref{eq-lowerby1} yield
	\begin{align}\label{eq-lowerby2}
	\sum_{i=1}^n \pi_i^{-1}\sum_{j=1}^n \sum_{\ell=1}^n B_{ij}\nu_j &B_{i\ell}\nu_\ell
	\left\|\frac{y_j}{\nu_j} - \frac{y_\ell}{\nu_\ell} \right\|^2 \ge \frac{\min^2(\nu) \,\left(\min(B^+)\right)^2}{\max(\pi)\, \mathsf{D}(\bbG)\mathsf{K}(\bbG)}\,
	\frac{1}{2}\sum_{j=1}^n\sum_{\ell=1}^n\left\|\frac{y_j}{\nu_j} - \frac{y_\ell}{\nu_\ell} \right\|^2 \nonumber \\
	&\ge 
	\frac{\min^2(\nu)\,\left(\min(B^+)\right)^2}{\max^2(\nu) \max(\pi)\, \mathsf{D}(\bbG)\mathsf{K}(\bbG)}\,
	\frac{1}{2}
	\sum_{j=1}^n\sum_{\ell=1}^n\nu_j\nu_\ell\left\|\frac{y_j}{\nu_j} - \frac{y_\ell}{\nu_\ell} \right\|^2.
	\end{align}	
	To express the last term, since $\la\nu,\1\ra=1$, we 
    we apply relation~\eqref{eq-aver-disp}
	with $\g_i=\nu_i$ and $u_i=y_i/\nu_i$ for all $i\in[n]$, and thus obtain
	\begin{align*}
	\frac{1}{2}\sum_{j=1}^n\sum_{\ell=1}^n\nu_j\nu_\ell
	\left\|\frac{y_j}{\nu_j} - \frac{y_\ell}{\nu_\ell} \right\|^2
	&=\sum_{i=1}^n \nu_i \left\|\frac{y_i}{\nu_i} - \sum_{\ell=1}^n y_\ell\right\|^2.\end{align*}
	By combining the preceding relation with inequality~\eqref{eq-lowerby2}, 
	and by substituting the resulting lower bound back in~\eqref{eq-sum2},
	we obtain
	\begin{align*}
	\sum_{i=1}^n\pi_i \left\|\frac{w_i}{\pi_i} - \sum_{\ell=1}^m y_\ell\right\|^2
	\le\left(1 -   \frac{\min^2(\nu)\,\left(\min(B^+)\right)^2}
	{\max^2(\nu) \max(\pi)\, \mathsf{D}(\bbG)\mathsf{K}(\bbG)}\right)
	\sum_{i=1}^n \nu_i \left\|\frac{y_i}{\nu_i} - \sum_{\ell=1}^n y_\ell\right\|^2.
	\end{align*}	
	which yields the desired relation after taking the square roots.
	\end{proof}
	
	The third quantity that we use to capture the behavior of the $AB$/Push-Pull method is the
	$\pi_k$-weighted dispersion of the scaled vectors 
	$y_1^k/[\pi_k]_1,\ldots,y_n^k/[\pi_k]_n$, i.e,
	\begin{equation}\label{eq-scaledy}
	S(\by^k,\pi_k)=\sqrt{\sum_{i=1}^n [\pi_k]_i \left\|\frac{y_i^k}{[\pi_k]_i} - \sum_{\ell=1}^n y_\ell^k\right\|^2},
	\end{equation}
	where $\pi_k$ is the stochastic vector defined in~\eqref{eq-pik}, $y_i^k$ are the directions used in
	method~\eqref{eq-met} at time $k$, and $\by^k=(y_1^k,\ldots,y_n^k)$. We note that $S(\by^k,\pi_k)$ can also be interpreted through the $\pi_k$-induced norm in the Cartesian product space $\re^p\times\cdots\times\re^p$. Specifically, using the definition of the iterates $y_i^{k+1}$ in~\eqref{eq-y}, we express 
	$y_i^{k+1}$ as follows:
	\begin{equation}\label{eq-yandw}
	y_i^{k+1}=w_i^k +\nabla f_i(x_i^{k+1})-\nabla f_i(x_i^k),
	\qquad\hbox{with}\quad w_i^k=\sum_{j=1}^n [B_k]_{ij} y_j^k.\end{equation}
	By defining 
	$\bw^k=(w_1^k,\ldots,w_n^k)$ and $\bg^k=(\nabla f_1(x_1^k),\ldots,\nabla f_n(x_n^k))$, we have
	\begin{equation}\label{eq-ycomp}
	\by^{k+1}=\bw^k+ \bg^{k+1}-\bg^k\qquad\hbox{for all }k\ge0.
	\end{equation}
	Viewing $\by^{k+1}$
	as the matrix with columns $y_i^{k+1}$, and similarly $\bw^k$ and $\bg^k$,
	we can write
	\begin{equation}\label{eq-ymatrix}
	\by^{k+1}\diag^{-1}(\pi_{k+1})=\bw^k\diag^{-1}(\pi_{k+1})
	+ (\bg^{k+1}-\bg^k)\diag^{-1}(\pi_{k+1})~~\hbox{for all }k\ge0,
	\end{equation}
	where $\diag(u)$ is the diagonal matrix with the vector $u$ entries on its diagonal.
	With this alternative view of the method, we have
	\begin{equation}\label{eq-s-norm}
	S(\by^k,\pi_k)=\|\by^k\diag^{-1}(\pi_k) - \bs^k\|_{\pi_k}
	~~ \hbox{ with }	\bs^k=(s^k,\ldots,s^k), \ s^k=\sum_{\ell=1}^n y_\ell^k.
	\end{equation}
	We provide the recursive relation for $S(\by^k,\pi_k)$ 
	in the following proposition.
		
	\begin{proposition}\label{prop-ycontract}
    Let Assumptions~\ref{asm-graphs}-\ref{asm-functions} hold,
    we have for all $k\ge0$,
	\begin{equation*}
	\|\by^k\|_{\pi_k^{-1}}
	=\sqrt{S^2(\by^k,\pi_k) +\left\|\sum_{\ell=1}^n y_\ell^k\right\|^2}
	\le S(\by^k,\pi_k) +\left\|\sum_{\ell=1}^n y_\ell^k\right\|,\end{equation*}
	\vspace{-0.5cm}
	\begin{align*}
	S(\by^{k+1}\!,\pi_{k+1}) \!\le\! \tau_k S(\by^k\!,\pi_k) \!+\! \alpha L r_k \|\by^k\|\!+\!L r_k\! \left(\!c_k\sqrt{\!\frac{1}{\min(\phi_{k+1})}}\!+\!\sqrt{\!\frac{1}{\min(\phi_k)}} \right) \!D(\bx^k,\phi_k).
	\end{align*}
	Here, the scalars $r_k>0$ and $\tau_k\in(0,1)$ 
	are given by
	\[r_k=\sqrt{n} + \frac{1}{\sqrt{\min(\pi_{k+1})}},
	\qquad
	\tau_k =\sqrt{1 -   \frac{\min^2(\pi_k)\,b^2}
	{\max^2(\pi_k) \max(\pi_{k+1})\, \mathsf{D}(\bbG_k)\mathsf{K}(\bbG_k)} },\]
	where $\phi_k$ and $\pi_k$ are the stochastic vectors associated with the matrices $A_k$ and $B_k$.
	\end{proposition}
	\begin{proof}
	Firstly, we note that under given assumptions, by Lemma~\ref{lem-bmatrices} we have that the stochastic vectors $\pi_k$, $k\ge0$, defined in~\eqref{eq-pik}, all have positive entries. 
	Noting that
	\[\|\by^k\|^2_{\pi_k^{-1}}
	=\sum_{i=1}^n \frac{\|y_i^k\|^2}{[\pi_k]_i} 
	=\sum_{i=1}^n [\pi_k]_i\,\left\|\frac{y_i^k}{[\pi_k]_i} \right\|^2,\]
	and using relation~\eqref{eq-gen-aver} for the weighted average of vectors, where $\g_i=[\pi_k]_i$ and
	$u_i=y_i^k/[\pi_k]_i$ for all $i$, and $u=0$, we obtain
	\[\sum_{i=1}^n [\pi_k]_i\,\left\|\frac{y_i^k}{[\pi_k]_i} \right\|^2
	=\sum_{i=1}^n [\pi_k]_i\left\|\frac{y_i^k}{[\pi_k]_i} -\sum_{\ell=1}^n y_\ell^k\right\|^2 
	+\left\|\sum_{\ell=1}^n y_\ell^k\right\|^2
	=S^2(\by^k,\pi_k) +\left\|\sum_{\ell=1}^n y_\ell^k\right\|^2,\]
	where the last equality is obtained from the definition of $S(\by^k,\pi_k)$ (see~\eqref{eq-scaledy}).
	Hence, 
	\begin{equation}\label{eq-ypinorm}
	\|\by^k\|_{\pi_k^{-1}}
	=\sqrt{S^2(\by^k,\pi_k) +\left\|\sum_{\ell=1}^n y_\ell^k\right\|^2}
	\le S(\by^k,\pi_k) +\left\|\sum_{\ell=1}^n y_\ell^k\right\|,\end{equation}
	where the inequality is obtained by using $\sqrt{a+b}\le \sqrt{a}+\sqrt{b}$, 
	which is valid for any two scalars
	$a,b\ge0$. Thus, we	have established the first relation of the proposition.
	
	We next proceed to show the relation for $S(\by^{k+1},\pi_{k+1})$.
	By~\eqref{eq-ymatrix}, we have
	\begin{equation*}
	\by^{k+1}\diag^{-1}(\pi_{k+1})=\bw^k\diag^{-1}(\pi_{k+1})
	+ (\bg^{k+1}-\bg^k)\diag^{-1}(\pi_{k+1})\qquad\hbox{for all }k\ge0.
	\end{equation*}
	By subtracting the vector 
	$\bs^{k+1}=(s^{k+1},\ldots,s^{k+1})$, where $s^{k+1}=\sum_{\ell=1}^n y_\ell^{k+1}$,
	from both sides of the preceding relation, we have for all $k\ge0,$
	\begin{equation*}
	\by^{k+1}\diag^{-1}(\pi_{k+1})-\bs^{k+1} 
	=\bw^k\diag^{-1}(\pi_{k+1})-\bs^k +(\bs^k-\bs^{k+1})+ (\bg^{k+1}-\bg^k)\diag^{-1}(\pi_{k+1}).
	\end{equation*}
	By taking $\pi_{k+1}$-induced norm on both sides of the preceding equality and 
	by using relation between $S(\by^{k+1},\pi_{k+1})$ and the $\pi_{k+1}$-induced norm
	(see~\eqref{eq-s-norm}),
	we have that
	\begin{align}\label{eq-srel1}
	&S(\by^{k+1},\pi_{k+1})
	=
	\|\bw^k\diag^{-1}(\pi_{k+1})\!-\bs^k +(\bs^k\!-\bs^{k+1})+ (\bg^{k+1}\!-\bg^k)\diag^{-1}(\pi_{k+1})\|_{\pi_{k+1}}\nonumber \\
	\!&\le\! \|\bw^k\diag^{-1}(\pi_{k+1})\!-\!\bs^k\|_{\pi_{k+1}}\!+\!\|\bs^k\!-\!\bs^{k+1}\|_{\pi_{k+1}}\!+\|(\bg^{k+1}\!-\!\bg^k)\diag^{-1}(\pi_{k+1})\|_{\pi_{k+1}}.
	\end{align}
	
	We next consider $\|\bw^k\diag^{-1}(\pi_{k+1})-\bs^k\|_{\pi_{k+1}}$, for which 
	by using the definitions of $\bw^k$ and $\bs^k$, i.e.,
	$ \bw^k=(w_1^k,\ldots,w_n^k)$ and $\bs^k=(s^k,\ldots, s^k)$,
	we have that 
	\[\|\bw^k\diag^{-1}(\pi_{k+1})-\bs^k\|_{\pi_{k+1}}
	=\sqrt{\sum_{i=1}^n[\pi_{k+1}]_i \left\|\frac{w_i^k}{[\pi_{k+1}]_i}-s^k\right\|^2},\]
	where $s^k=\sum_{\ell=1}^n y_\ell^k$ (see~\eqref{eq-s-norm}). 
	We now apply Lemma~\ref{lem-basic-ycontract}
	with the following identification $\bbG=\bbG_k$, $B=B_k$, $\pi=\pi_{k+1}$, and $\nu=\pi_k$, 
	which yields
	\begin{align*}
	\sum_{i=1}^n[\pi_{k+1}]_i \left\|\frac{w_i^k}{[\pi_{k+1}]_i} - \sum_{\ell=1}^m y_\ell\right\|
	\le\tau_k 
	\sqrt{\sum_{i=1}^n [\pi_k]_i \left\|\frac{y_i}{[\pi_k]_i} - \sum_{\ell=1}^n y_\ell\right\|^2},
	\end{align*}	
	with
	$\tau_k=\sqrt{1 -   \frac{\min^2(\pi_k)\,b^2}
	{\max^2(\pi_k) \max(\pi_{k+1})\, \mathsf{D}(\bbG_k)\mathsf{K}(\bbG_k)} },$
	where we use $\min(B_k^+)\ge b$.
	Hence,
	\[\|\bw^k\diag^{-1}(\pi_{k+1})-\bs^k\|_{\pi_{k+1}}
	\le \tau_k\,
	\sqrt{\sum_{i=1}^n [\pi_k]_i \left\|\frac{y_i}{[\pi_k]_i} - \sum_{\ell=1}^n y_\ell\right\|^2}
	=\tau_k\,S(\by^k,\pi_k),\]
	where the equality follows from the definition of $S(\by^k,\pi_k)$ in~\eqref{eq-scaledy}.
	Thus, by substituting the preceding relation back in~\eqref{eq-srel1}, we have
	\begin{align}\label{eq-srel2}
	S(\by^{k+1}\!,\pi_{k+1})
	\!\le \tau_k S(\by^k,\pi_k)
	+\|\bs^k\!-\!\bs^{k+1}\|_{\pi_{k+1}}
	+\|(\bg^{k+1}\!-\!\bg^k)\diag^{-1}(\pi_{k+1})\|_{\pi_{k+1}}.
	\end{align}
	
	Next, we consider the term $\|\bs^k-\bs^{k+1}\|_{\pi_{k+1}}$ in~\eqref{eq-srel2},
	for which we have
	\[\|\bs^k-\bs^{k+1}\|_{\pi_{k+1}}=\sqrt{\sum_{i=1}^n [\pi_{k+1}]_i\|s^{k+1}-s^k\|^2}
	=\|s^{k+1}-s^k\|,\]
	where the last equality follows since the vector $\pi_{k+1}$ is stochastic.
	By the definition of $s^k$ in~\eqref{eq-s-norm}, we have 
	$s^k=\sum_{\ell=1}^n y_\ell^k$ . Since $B_k$ is column-stochastic,
	by Lemma~\ref{lem-sumgrad}(a) we further have 
	$\sum_{\ell=1}^n y_\ell^k=\sum_{\ell=1}^n \nabla f_\ell(x_\ell^k)$, implying that 
	\[\|\bs^k-\bs^{k+1}\|_{\pi_{k+1}}
	=\left\|\sum_{\ell=1}^n \left(\nabla f_\ell(x_\ell^{k+1})- \nabla f_\ell(x_\ell^k) \right)\right\|
	\le \sum_{\ell=1}^n \|\nabla f_\ell(x_\ell^{k+1})- \nabla f_\ell(x_\ell^k)\|.\]
	By using the Lipschitz continuity of the gradients $\nabla f_i$, we obtain
	\begin{align}\label{eq-est1}
	\|\bs^k-\bs^{k+1}\|_{\pi_{k+1}}
	\le L\sum_{\ell=1}^n \|x_\ell^{k+1} - x_\ell^k\|\le L\sqrt{n}\|\bx^{k+1}-\bx^k\|.\end{align}	
	For the term $\|(\bg^{k+1}-\bg^k)\diag^{-1}(\pi_{k+1})\|_{\pi_{k+1}}$ in relation~\eqref{eq-srel2},
    we have
	\[\|(\bg^{k+1}-\bg^k)\diag^{-1}(\pi_{k+1})\|_{\pi_{k+1}}
	=\|\bg^{k+1}-\bg^k\|_{\pi_{k+1}^{-1}}
	=\sqrt{\sum_{i=1}^n   \frac{\|\nabla f_i(x_i^{k+1})- \nabla f_i(x_i^k)\|^2}{[\pi_{k+1}]_i}}. \]
	By the Lipschitz continuity property of the gradients $\nabla f_i(\cdot)$, 
	we obtain
	\begin{align}\label{eq-est2}
	\|(\bg^{k+1}\!\!-\!\bg^k)\diag^{-1}(\pi_{k+1})\|_{\pi_{k+1}}
	\!\le\! L\sqrt{\!\sum_{i=1}^n \!\frac{\|x_i^{k+1} \!\!-\! x_i^k\|^2}{[\pi_{k+1}]_i} }
	\!\le\! \frac{L}{\sqrt{\min(\pi_{k+1})}}\,\|\bx^{k+1}\!\!-\!\bx^k\|.
	\end{align}
	Now, we combine the estimates in~\eqref{eq-est1} and~\eqref{eq-est2} with relation~\eqref{eq-srel2}
	and obtain that
	\begin{align*}
	S(\by^{k+1},\pi_{k+1})
	\le \tau_k\, S(\by^k,\pi_k)+L r_k\|\bx^{k+1}-\bx^k\|,
	\end{align*}
	where $r_k=\sqrt{n}+ \frac{1}{\sqrt{\min(\pi_{k+1})}}$.
	The desired relation follows from the preceding relation and the estimate for $\|\bx^{k+1}-\bx^k\|$
	in Lemma~\ref{lem-xdiff} (see \eqref{eq-xdiff}).
	\end{proof}
	
	\section{Convergence Results}\label{sec-conver}
	In this section, we combine the results obtained in Sections~\ref{ssec-waver}--\ref{ssec-yvar}
	to obtain a composite relation for the main quantities of interest.

	\subsection{Composite Relation} \label{sec-comrel}
	We first give the relations in a compact form by defining the vector $V_k$ as follows
	\begin{align}\label{eq-vk}
	V_k=\left(\|\hat x^k -x^*\|,D(\bx^k,\phi_k),S(\by^k,\pi_k)\right)^{\T},
	\end{align}
	which we recall below for convenience:
	\begin{align*}
		D(\bx^k,\phi_k)= \sqrt{\sum_{i=1}^n[\phi_k]_i\|x_i^k-\hat x^k\|^2},\quad S(\by^k,\pi_k) = \sqrt{\sum_{j=1}^n[\pi_k]_j\left\|\frac{y_j^k}{[\pi_k]_j} - \sum_{\ell=1}^n y_\ell^k\right\|^2},
	\end{align*}
	where $\hat x^k= \sum_{i=1}^n[\phi_k]_i x_i^k$ 
	and $x^*$ is the solution of the problem~\eqref{eq-problem}.
	Using Propositions~\ref{prop-waverx}(c),~\ref{prop-xcontract}, and~\ref{prop-ycontract}, we establish a relation between $V_{k+1}$ and $V_k$ that will involve the constants $c_k$, and $\tau_k$
	from Proposition~\ref{prop-xcontract} and
	Proposition~\ref{prop-ycontract}, given by
	\begin{subequations}\label{eq-constants}
	\begin{align}
	&\!\!\!\!q_k(\a) = \max\big\{|1-\a n \min(\pi_k)\mu|,|1-\a n \min(\pi_k)L|\big\},~
	r_k=\sqrt{n} + \frac{1}{\sqrt{\min(\pi_{k+1})}},
	\label{eq-qandrk}\\
	&\!\!\!\!c_k=\sqrt{1 \!-\! \frac{\min(\phi_{k+1}) \,a^2}
	{\max^2(\phi_k)\mathsf{D}(\bbG_k)\mathsf{K}(\bbG_k)} },
	\tau_k =\sqrt{1 \!-\! \frac{\min^2(\pi_k)\,b^2}
	{\max^2(\pi_k) \max(\pi_{k+1}\!) \mathsf{D}(\bbG_k)\mathsf{K}(\bbG_k)} }.
	\label{eq-ckandtauk}
	\end{align}
	\end{subequations}
	
	For the vector $V_k$ we  have the following result.
			
	\begin{proposition}\label{prop-main}
    Let Assumptions~\ref{asm-graphs}-\ref{asm-strconv} hold. Consider the iterates produced by the $AB$/Push-Pull method in~\eqref{eq-met} 
	with the stepsize $\alpha\in(0,2(nL)^{-1})$, we have
	\[V_{k+1}\le M_k(\alpha)V_k\qquad \hbox{for all $k\ge0$},\]
	where $M_k(\alpha)$ is the matrix given by
	\begin{equation*}\label{eq-matrixm}
	 M_k(\a)=
	 \left[\begin{array}{ccc}
	 q_k(\a) & \a L \sqrt{n} \varphi_k & \a \cr
	 \hbox{}\cr
	 \a L \g_k \sqrt{n} \varphi_k & c_k + \a L \g_k \sqrt{n} \varphi_k  & \a \g_k\cr
	 \hbox{}\cr
	 \a L^2r_k \sqrt{n} \varphi_k & Lr_k(c_k\varphi_{k+1}+\varphi_k) + \a L^2r_k \sqrt{n} \varphi_k & \tau_k + \a Lr_k \end{array}\right].
	\end{equation*}
	with
	$\displaystyle \g_k=\sqrt{\max_{j\in[n]} ([\phi_{k+1}]_j[\pi_k]_j)} $,
	$\varphi_k=\sqrt{\frac{1}{\min(\phi_{k})}}$, and $q_k(\a)$, $c_k$, $\tau_k$, and $r_k$ as in~\eqref{eq-constants}. 
	\end{proposition}
	\begin{proof}
	The first row of $M_k(\a)$ is given by Proposition~\ref{prop-waverx}(b) when $\alpha\in(0,2(nL)^{-1})$.
	Next, we consider the relation for $D(\bx^{k+1},\phi_{k+1})$. 
	By Proposition~\ref{prop-xcontract}, we have that 
	\begin{equation}\label{eq-fin1}
	D(\bx^{k+1},\phi_{k+1})\le 
	c_k D(\bx^k,\phi_k)
	+ \a \sqrt{\sum_{i=1}^n [\phi_{k+1}]_i \left\|y_i^k - \sum_{j=1}^n[\phi_{k+1}]_j y_j^k\right\|^2}.
	\end{equation}
	Using relation~\eqref{eq-gen-aver} with $\g_i=[\phi_{k+1}]_i$, $u_i=y_i^k$ for all $i$ and $u=0$,
	it follows that
	\[\sum_{i=1}^n [\phi_{k+1}]_i \left\|y_i^k - \sum_{j=1}^n[\phi_{k+1}]_j y_j^k\right\|^2
	\le \sum_{i=1}^n [\phi_{k+1}]_i \left\|y_i^k \right\|^2.\]
	By multiplying and dividing each term in the summation on the right hand side with $[\pi_k]_i$,
	we find that 
	\[\sum_{i=1}^n [\phi_{k+1}]_i \left\|y_i^k -\! \sum_{j=1}^n[\phi_{k+1}]_j y_j^k\right\|^2 \!\!
	\le \sum_{i=1}^n [\phi_{k+1}]_i[\pi_k]_i \frac{\left\|y_i^k \right\|^2}{[\pi_k]_i} 
	\le \max_{j\in[n]} ([\phi_{k+1}]_j[\pi_k]_j) \sum_{i=1}^n\frac{\left\|y_i^k \right\|^2}{[\pi_k]_i}.\]
	Therefore, by taking the square roots on both sides of the preceding relation, we obtain
	\[\sqrt{\sum_{i=1}^n [\phi_{k+1}]_i \left\|y_i^k - \sum_{j=1}^n[\phi_{k+1}]_j y_j^k\right\|^2}
	\le \sqrt{\max_{j\in[n]} ([\phi_{k+1}]_j[\pi_k]_j)} \sqrt{\sum_{i=1}^n\frac{\left\|y_i^k \right\|^2}{[\pi_k]_i}}
	=\g_k \|\by^k\|_{\pi_k^{-1}},\]
	where the equality follows upon using $\g_k=\sqrt{\max_{j\in[n]} ([\phi_{k+1}]_j[\pi_k]_j)}$ and the definition of $\|\by^k\|_{\pi_k^{-1}}$.
	Substituting the preceding estimate back in relation~\eqref{eq-fin1}, we find that,
	\begin{equation*}
	D(\bx^{k+1},\phi_{k+1})\le 
	c_k D(\bx^k,\phi_k)
	+ \a \g_k \|\by^k\|_{\pi_k^{-1}} ~\text{for all}~ k\ge0.\end{equation*}
	Using the preceding relation, the relation
	$\|\by^k\|_{\pi_k^{-1}}
	\le S(\by^k,\pi_k) +\left\|\sum_{\ell=1}^n y_\ell^k\right\|$ established in 
	Proposition~\ref{prop-ycontract}, and the following relation from Lemma~\ref{lem-xdiff}	
	\begin{align} \label{eq-sumy}
	\left\|\sum_{i=1}^n y_i^k\right\|
	\le L\sqrt{\frac{n}{\min(\phi_k)}} \left(\|\hat x^k - x^*\| + D(\bx^k,\phi_k)\right),
	\end{align}
	we obtain the desired relation for $D(\bx^{k+1},\phi_{k+1})$ (given by the second row of $M_k(\a)$).
	
	Lastly, the relation for $S(\by^{k+1},\pi_{k+1})$ comes from  Proposition~\ref{prop-ycontract}.
	For the quantity $\|\by^k\|$, using the vector-induced norm property in~\eqref{eq-stochvec-norm} and the fact that the vector $\pi_k$ is stochastic for all $k$, we have $
	\|\by^k\|\le \|\by^k\|_{\pi_k^{-1}}$. Upon using the relations $\|\by^k\|_{\pi_k^{-1}}
	\le S(\by^k,\pi_k) +\left\|\sum_{\ell=1}^n y_\ell^k\right\|$ established in 
	Proposition~\ref{prop-ycontract} and \eqref{eq-sumy}, we obtain
	\begin{align*}
	S(\by^{k+1},\pi_{k+1})&\le \left(\tau_k+\a Lr_k\right)\, S(\by^k,\pi_k) + \a L^2 r_k \sqrt{\frac{n}{\min(\phi_{k})}}\|\hat x^{k} -x^*\|\cr
	&+L r_k \left(c_k\sqrt{\frac{1}{\min(\phi_{k+1})}}+\sqrt{\frac{1}{\min(\phi_k)}} +\a L\sqrt{\frac{n}{\min(\phi_{k})}} \right) D(\bx^k,\phi_k),
	\end{align*}
	which gives the third row of $M_k(\a)$.
	\end{proof}
	
	\subsection{Convergence Result and Range for the Stepsize}
	From Proposition~\ref{prop-main}, to prove that $V_k$ tends to 0 at a geometric rate, it is sufficient to show that $M_k(\a)\le M(\a)$ for some matrix $M(\a)$, and then choose a suitable stepsize $\a\in(0,2(nL)^{-1})$
	such that the eigenvalues of $M(\a)$ are inside the unit circle, i.e., the spectral radius of $M(\a)$ is less than 1.
	
	We now determine an upper bound matrix $M(\a)$ for $M_k(\a )$. Let $c\in(0,1)$, $\tau\in(0,1)$, $r$, and $\varphi$, be upper bounds for $c_k$, $\tau_k$, $r_k$, and $\varphi_k$, respectively, i.e.,
	\begin{align}\label{eq-maxconst}
	&\max_{k\ge0}c_k\le c,\quad \max_{k\ge0}\tau_k\le \tau,\quad \max_{k\ge0}r_k\le r,\quad \max_{k\ge0} \varphi_k \le \varphi.
	\end{align}
	For the quantity $q_k(\a)$ as in \eqref{eq-qandrk}, when $\a\in(0,2(nL+n\mu)^{-1})$, we have $q_k(\a)=1-\a n \min(\pi_k)\mu <1$. Let $\sigma$ be a lower bound for $\min(\pi_k)$, $k\ge0$, corresponding to the graph sequence $\{\mathbb{G}_k\}$. In the most general case of graph sequences, by Lemma~\ref{lem-bmatrices} we have that  $\sigma \le \min_{k\ge0}\{\min(\pi_k)\}$ with $\sigma \ge \tfrac{b^n}{n} >0$. Thus, we have the following upper bound for $q_k(\a)$:
	\begin{align}\label{eq-maxconst2}
	\max_{k\ge0}q_k(\a) \le 1-\a n \sigma \mu\in (0,1)\qquad \text{ where }\quad \sigma \le \min_{k\ge0}\{\min(\pi_k)\}.
	\end{align}
	We notice also that $\g_k=\max_{j\in[n]} ([\phi_{k+1}]_j[\pi_k]_j)\le 1$ since $\phi_k$ and $\pi_k$ are stochastic vectors.
	Using these upper-bounds, for $\a\in(0,2(nL+n\mu)^{-1})$, we have $M_k(\a )\le M(\a)$, for all $k \ge 0$, with the matrix $M(\a)$ given by
    \begin{equation}\label{eq-gmatrixm}
	M(\a)=
	\left[\begin{array}{ccc}
	 1-\a n \sigma \mu & \a L \sqrt{n} \varphi & \a \cr
	 \hbox{}\cr
	 \a L\sqrt{n}\varphi & c + \a L\sqrt{n}\varphi  & \a \cr
	 \hbox{}\cr
	 \a L^2r \sqrt{n} \varphi & Lr(1+c)\varphi + \a L^2r \sqrt{n} \varphi & \tau + \a Lr \end{array}\right].
	\end{equation}
	
	\begin{proposition}
	Let Assumptions~\ref{asm-graphs}-\ref{asm-strconv} hold. Consider the iterates produced by the method in~\eqref{eq-met} and the notation in \eqref{eq-maxconst}-\eqref{eq-maxconst2}. If the stepsize $\a>0$ is chosen such that 
	\begin{align}\label{eq:alpha-range}
    \a \le \min \left\{\dfrac{1-c}{L\sqrt{n}\varphi},~\dfrac{1-\tau}{Lr},~\dfrac{n\sigma\mu(1-\tau)(1-c)}{\eta},\dfrac{2}{n(L+\mu)}\right\},
    \end{align}
    where $\eta = L(n\sigma\mu+L\sqrt{n}\varphi)\left((1+c)r\varphi+(1-c)r+(1-\tau)\sqrt{n}\varphi\right)>0$. Then, 
    \[\lim_{k\to\infty}\|x_i^k -x^*\|=0,\qquad \text{for all } i\in [n].\]
	\end{proposition}
	\begin{proof}
	Recall that by Proposition~\ref{prop-main}, we have $V_{k+1}\le M_k(\a)V_k$, for all $k\ge 0$. Therefore, with the matrix $M(\a)$ defined as in \eqref{eq-gmatrixm}, we have
	\begin{equation}\label{eq-vkrel}
	    V_{k+1}\le M(\a)V_k, \qquad \hbox{for all $k\ge0$}.
	\end{equation}
	Thus, $\|\hat x^k -x^*\|,~D(\bx^k,\phi_k)$ and $S(\by^k,\pi_k)$ all converge to 0 linearly at rate $\mathcal{O}\left(\rho_M^k\right)$ if the spectral radius $\rho_{M(\a)}$ of $M(\a)$ satisfies $\rho_{M(\a)}<1$. By Lemma~8 of~\cite{pshi21}, we will have $\rho_{M(\a)}<1$ if all diagonal entries of $M(\a)$ are less than 1 and $\det(\mbi-M(\a))>0$, where 
    \begin{align*}
	\det(M(\a)-\mbi)&=
	\left| \begin{array}{ccc}
	-\a n \sigma \mu & \a L \sqrt{n} \varphi & \a \cr
	\hbox{}\cr
	\a L\sqrt{n}\varphi & c + \a L\sqrt{n}\varphi-1  & \a \cr
	\hbox{}\cr
	\a L^2r \sqrt{n} \varphi & Lr(1+c)\varphi + \a L^2r \sqrt{n} \varphi & \tau + \a Lr-1 \end{array}\right|.
	\end{align*}
	Hence,
	\begin{align*}
	\det(M(\a)-\mbi)=\a\left[\a \eta-n\sigma\mu(1-\tau)(1-c)\right],
	\end{align*}
	where $\eta = L(n\sigma\mu+L\sqrt{n}\varphi)\left((1+c)r\varphi+(1-c)r+(1-\tau)\sqrt{n}\varphi\right)>0$ since $c<1$ and $\tau<1$. It remains to choose $\a\in(0,2(nL+n\mu)^{-1})$ so that the following conditions are satisfied
    $$\begin{cases} 
    c+\a L\sqrt{n}\varphi<1 \\
    \tau + \a Lr< 1\\
    \a \eta-n\sigma\mu(1-\tau)(1-c)<0.
    \end{cases}$$
    Solving the preceding system of inequalities yields the range in \eqref{eq:alpha-range}.
    \end{proof}
    
    \begin{remark}
    We can relax Assumption~\ref{asm-graphs} by considering a $C$-strongly-connected graph sequence, i.e., there exists some integer, $C\ge 1$ such that the graph with edge set
    $\mathcal{E}^C_k=\bigcup_{i=kC}^{(k+1)C-1}\mathcal{E}_i$
    is strongly connected for every $k\ge 0$. In this case, the more general results of Lemma~\ref{lem-amatrices} and Lemma~\ref{lem-bmatrices} state that there exist stochastic vector sequences $\{\phi_k\}$ and $\{\pi_k\}$, such that for all $k \ge 0$, 
    $$\phi'_{k+C}\left(A_{k+C-1}\ldots A_{k+1}A_{k}\right)=\phi'_k ~\text{and}~ \pi_{k+C}=\left(B_{k+C-1}\ldots B_{k+1}B_{k}\right)\pi_k.$$
  Furthermore,
    $$[\phi_k]_i\ge\dfrac{a^{nC}}{n} ~\text{and}~[\pi_k]_i\ge\dfrac{b^{nC}}{n} \quad \text{for all } i\in[n].$$
    With the use of these results, the rest of convergence analysis follows similarly to our analysis for strongly connected graphs, by noticing that contractions due to row- and column-stochastic matrices occur after time $k=C$. 
\end{remark}

    \section{Numerical Simulations}
    
	In this section, we evaluate the performance of the proposed algorithm through a sensor fusion problem over a network, as described in \cite{Jinming2018}. The estimation problem is given as follows
	$$\min_{x\in \mathbb{R}^p}~ \sum\limits_{i=1}^n \left(\|z_i-H_ix\|^2+\lambda_i \|x\|^2\right),$$
    where $x$ is the unknown parameter to be estimated, $H_i \in \mathbb{R}^{s\times p}$ represents the measurement matrix, $z_i=H_ix+\omega_i \in \mathbb{R}^{s}$ is the noisy observation of sensor $i$ with some noise $\omega_i$ and $\lambda_i$ is the regularization parameter for the local cost function of sensor $i$.
    
    As in \cite{pshi21}, we set $n=20$, $p=20$ and $s=1$ so that each local cost function is ill-conditioned, requiring the coordination among agents to achive fast convergence. The measurement matrix $H_i$ is generated from a uniform distribution in the unit $\re^{s\times p}$ space which is then normalized such that its Lipschitz constant is equal to $1$. The noise $\omega_i$ follows an i.i.d. Gaussian process with zero mean and unit variance $\mathcal{N}(0,1)$. The regularization parameter is chosen to be $\lambda_i=0.01$, for all $i\in [n]$, to ensure the strong convexity of the loss function.
    
    \sidecaptionvpos{figure}{c}
    \begin{SCfigure}[50][h!]
    \caption{Residuals plot}\label{fig:residual}
    ~~~\includegraphics[width=0.55\textwidth]{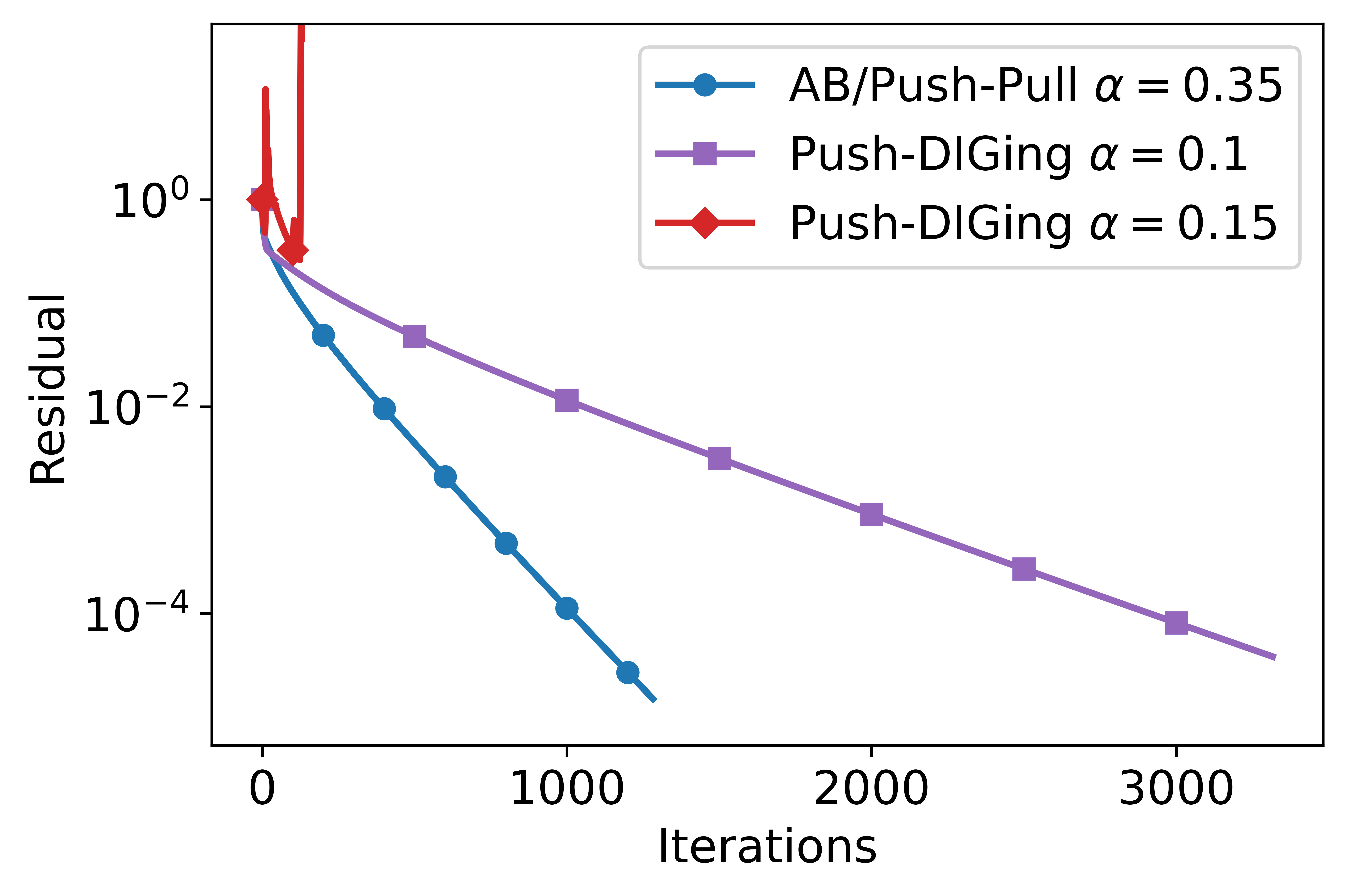}
    \end{SCfigure}
    
    We compare our proposed AB/Push-Pull algorithm against Push-DIGing \cite{nedic2017achieving}.
    The simulation is carried out over a random sequence of time-varying directed communication network 
    . The performance is compared in terms of the relative residual defined as $\tfrac{\|\bx^k-\bx^*\|^2_2}{\|\bx^0-\bx^*\|^2_2}$. Figure~\ref{fig:residual} illustrates the performance of the above algorithms under a randomly generated time-varying network. As discussed in \cite{pshi21}, AB/Push-Pull allows for much larger value of the stepsize compared to Push-DIGing and it converges faster especially for ill-conditioned problems and when graphs are not well balanced.

	\section{Conclusions}\label{sec: conclusion}
In this paper, we study a distributed optimization problem over a time-varying directed communication network. We consider the $AB$/Push-Pull gradient-based method where each node maintains estimates of the optimal decision variable and the average gradient of the agents’ objective functions. The information about the decision variable is pushed to its neighbors, while the information about the gradients is pulled from its neighbors using both row- and column-stochastic weights simultaneously. We explore the contractive properties of the iterates produced by the method, which are inherited from the use of the mixing terms and the fact that the mixing matrices are compliant with a directed strongly connected graph. We prove that the algorithm converges linearly to the global minimizer for smooth and strongly convex cost functions. The convergence result is derived based on the choice of appropriate stepsize values for which explicit upper bounds are provided in terms of the properties of the cost functions, the mixing matrices, and the graph connectivity structure.


\begin{thebibliography}{19}

\bibitem{Bazerque2010}
J. Bazerque and G. Giannakis, Distributed Spectrum Sensing for Cognitive Radio Networks by Exploiting Sparsity, IEEE Trans. Signal Process. 58 (2010), pp. 1847--1862.

\bibitem{Duchi2012}
J. Duchi, A. Agarwal, and M. Wainwright, Dual Averaging for Distributed Optimization:
Convergence Analysis and Network Scaling, IEEE Trans. Autom. Control 57 (2012), pp.
592--606.

\bibitem{Huan2022}
H. Gao, Y. Wang, and A. Nedi{\'c}, Dynamics based Privacy Preservation in Decentralized Optimization. arXiv preprint arXiv:2207.05350 (2022).

\bibitem{Gharesifard2012}
B. Gharesifard and J. Cortes, Distributed Strategies for Generating Weight-Balanced and
Doubly Stochastic Digraphs, European Journal of Control 18 (2012), pp. 539--557.

\bibitem{Huan2021}
H. Li and Z. Lin, Accelerated Gradient Tracking over Time-Varying Graphs for Decentralized Optimization. arXiv preprint arXiv:2104.02596 (2021).

\bibitem{mokhtari2017network}
A. Mokhtari, Q. Ling, and A. Ribeiro, Network Newton Distributed Optimization Methods,
IEEE Trans. Signal Process. 65 (2017), pp. 146--161.

\bibitem{Nedic2009}
A. Nedi{\'c} and A. Ozdaglar, Distributed Subgradient Methods for Multi-agent Optimization,
IEEE Trans. Autom. Control 54 (2009), pp. 48--61.

\bibitem{Nedic2011}
A. Nedi{\'c}, Asynchronous Broadcast-Based Convex Optimization over a Network, IEEE
Trans. Autom. Control 56 (2011), pp. 1337--1351.

\bibitem{nedic2015distributed}
A. Nedi{\'c} and A. Olshevsky, Distributed Optimization over Time-Varying Directed Graphs,
IEEE Trans. Autom. Control 60 (2015), pp. 601--615.

\bibitem{nedic2017achieving}
A. Nedi{\'c}, A. Olshevsky, and W. Shi, Achieving Geometric Convergence for Distributed
Optimization Over Time-Varying Graphs, SIAM J. Optim. 27 (2017), pp. 2597--2633.

\bibitem{nguyen2022distributed}
D.T.A. Nguyen, D.T. Nguyen, and A. Nedi{\'c}, Distributed Nash Equilibrium Seeking over Time-Varying Directed Communication Networks. arXiv preprint arXiv:2201.02323 (2022).

\bibitem{olshevsky2017linear}
A. Olshevsky, Linear Time Average Consensus and Distributed Optimization on Fixed
Graphs, SIAM J. Control Optim. 55 (2017), pp. 3990--4014.

\bibitem{Polyak}
B. Polyak, Introduction to Optimization, New York : Optimization Software, Inc., 1987.

\bibitem{Pu2020}
S. Pu, A Robust Gradient Tracking Method for Distributed Optimization over Directed
Networks, in 59th IEEE Conf. Decis. Control (2020), pp. 2335--2341.

\bibitem{pu2017flocking}
S. Pu and A. Garcia, A Flocking-based Approach for Distributed Stochastic Optimization,
Operations Research 1 (2018), pp. 267--281.

\bibitem{pu2018distributed}
S. Pu and A. Nedi{\'c}, A Distributed Stochastic Gradient Tracking Method, in IEEE
Conf. Decis. Control (CDC) (2018), pp. 963--968.

\bibitem{pshi21}
S. Pu, W. Shi, J. Xu, and A. Nedi{\'c}, Push--Pull Gradient Methods for Distributed Optimization in Networks, IEEE Trans. Autom. Control 66 (2021), pp. 1--16.

\bibitem{Qu2017}
G. Qu and N. Li, Harnessing Smoothness to Accelerate Distributed Optimization, IEEE Trans. Control. Netw. Syst. 5 (2017), pp. 1245--1260.

\bibitem{Rabbat2004}
M. Rabbat and R. Nowak, Distributed Optimization in Sensor Networks, in 3rd International Symposium on Information Processing in Sensor Networks (2004), pp. 20--27.

\bibitem{Raffard2004}
R. Raffard, C. Tomlin, and S. Boyd, Distributed Optimization for Cooperative Agents: Application to Formation Flight, in 43rd IEEE Conf. Decis. Control (CDC), Vol. 3. (2004), pp. 2453--2459.

\bibitem{Ram2009}
S. Ram, V. Veeravalli, and A. Nedi{\'c}, Distributed Non-Autonomous Power Control through Distributed Convex Optimization, in INFOCOM (2009), pp. 3001--3005.

\bibitem{Saadatniaki2020}
F. Saadatniaki, R. Xin, and U.A. Khan, Decentralized Optimization Over Time-Varying Directed Graphs With Row and Column-Stochastic Matrices, IEEE Trans. Autom. Control 65 (2020), pp. 4769--4780.

\bibitem{Shi2015_2}
W. Shi, Q. Ling, G. Wu, and W. Yin, A Proximal Gradient Algorithm for Decentralized Composite Optimization, IEEE Trans. Signal Process. 63 (2015), pp. 6013--6023.

\bibitem{Shi2015}
W. Shi, Q. Ling, G. Wu, and W. Yin, EXTRA: An Exact First-Order Algorithm for Decentralized Consensus Optimization, SIAM J. Optim. 25 (2015), pp. 944--966.

\bibitem{Wei2014}
W. Shi, Q. Ling, K. Yuan, G. Wu, and W. Yin, On the Linear Convergence of the ADMM in Decentralized Consensus Optimization, IEEE Trans. Signal Process. 62 (2014), pp. 1750--1761.

\bibitem{Srivastava2011}
K. Srivastava and A. Nedi{\'c}, Distributed Asynchronous Constrained Stochastic Optimization, IEEE Journal of Selected Topics in Signal Processing 5 (2011), pp. 772--790.

\bibitem{Stipanovic2002}
D. Stipanovic, G. Inalhan, R. Teo, and C. Tomlin, Decentralized Overlapping Control of A Formation of Unmanned Aerial Vehicles, in 41st IEEE Conf. Decis. Control, Vol. 3. (2002), pp. 2829--2835.

\bibitem{Tsianos2012}
K. Tsianos, S. Lawlor, and M. Rabbat, Push-Sum Distributed Dual Averaging for Convex Optimization, in 51st IEEE Conf. Decis. Control (2012), pp. 5453--5458.

\bibitem{Tsianos2012ML}
K.I. Tsianos, S. Lawlor, and M.G. Rabbat, Consensus-Based Distributed Optimization: Practical Issues and Applications in Large-Scale Machine Learning, in 50th Annual Allerton Conference on Communication, Control, and Computing (2012), pp. 1543--1550.

\bibitem{varagnolo2016newton}
D. Varagnolo, F. Zanella, A. Cenedese, G. Pillonetto, and L. Schenato, Newton-Raphson Consensus for Distributed Convex Optimization, IEEE Trans. Autom. Control 61 (2016),
pp. 994--1009.

\bibitem{Wang2022}
Y. Wang and A. Nedi{\'c}, Tailoring Gradient Methods for Differentially-Private Distributed
Optimization. arXiv preprint arXiv:2202.01113 (2022).

\bibitem{xi2018linear}
C. Xi, V.S. Mai, R. Xin, E.H. Abed, and U.A. Khan, Linear Convergence in Optimization over Directed Graphs with Row-Stochastic Matrices, IEEE Trans. Autom. Control (2018).

\bibitem{xi2018add}
C. Xi, R. Xin, and U.A. Khan, ADD-OPT: Accelerated Distributed Directed Optimization,
IEEE Trans. Autom. Control 63 (2018), pp. 1329--1339.

\bibitem{XinSahuKhanKar2019}
R. Xin, A.K. Sahu, U.A. Khan, and S. Kar, Distributed Stochastic Optimization with
Gradient Tracking over Strongly-connected Networks, in 58th IEEE Conf. Decis. Control (2019), pp. 8353--8358.

\bibitem{xin2018linear}
R. Xin and U.A. Khan, A Linear Algorithm for Optimization Over Directed Graphs With
Geometric Convergence, IEEE Contr. Syst. Lett. 2 (2018), pp. 315--320.

\bibitem{Xin2019FROSTFastRO}
R. Xin, C. Xi, and U.A. Khan, FROST—Fast Row-Stochastic Optimization with Uncoordinated Step-sizes, EURASIP J. Adv. Signal Process. (2019), pp. 1--14.

\bibitem{Jinming2018}
J. Xu, S. Zhu, Y.C. Soh, and L. Xie, Convergence of Asynchronous Distributed Gradient Methods Over Stochastic Networks, IEEE Trans. Autom. Control 63 (2018), pp. 434--448.


\end{thebibliography}
\end{document}